\documentclass[a4paper,10pt]{article}
\usepackage{courier}
\usepackage{blindtext} 
\usepackage{amsmath}
\usepackage{amssymb}
\usepackage[colorlinks=true,breaklinks]{hyperref}
\usepackage[hyphenbreaks]{breakurl}
\usepackage{xcolor}
\definecolor{c1}{rgb}{0,0,1}
\definecolor{c2}{rgb}{0,0.3,0.9}
\definecolor{c3}{rgb}{0.3,0.9}
\hypersetup{linkcolor={c1},citecolor={c2},urlcolor={c3}}
\usepackage{enumerate}
\usepackage{todonotes}
\usepackage{makeidx}
\makeindex
\usepackage[top=3cm,bottom=3cm,left=2.5cm,right=2.5cm]{geometry}
\usepackage{fancyhdr}

\newcommand{\bbR}{\mathbb{R}}
\newcommand{\R}{\bbR}

\DeclareMathOperator{\diver}{div}

\def\XXint#1#2#3{{\setbox0=\hbox{$#1{#2#3}{\int}$ }
\vcenter{\hbox{$#2#3$ }}\kern-.6\wd0}}

\usepackage{amsthm}
\theoremstyle{plain}
\newtheorem{theorem}{Theorem}[section]

\theoremstyle{definition}

\theoremstyle{lemma}
\newtheorem{lemma}[theorem]{Lemma}

\theoremstyle{Remark}
\newtheorem{Remark}[theorem]{Remark}

\theoremstyle{proposition}
\newtheorem{proposition}[theorem]{Proposition}

\theoremstyle{corollary}
\newtheorem{corollary}[theorem]{Corollary}

\theoremstyle{example}

\theoremstyle{assumption}
\newtheorem{assumption}[theorem]{Assumption}

\usepackage{mathrsfs}
\usepackage{systeme}
\usepackage{colonequals}
\usepackage{comment}

\begin{document}
\pagestyle{empty}
\title{Trend to equilibrium and hypoelliptic regularity for the relativistic  Fokker-Planck equation}
\author{Anton Arnold$^\ast$, Gayrat Toshpulatov\thanks{Institut f\"ur Analysis \& Scientific Computing, Technische Universit\"at Wien, Wiedner Hauptstr. 8-10, A-1040 Wien, Austria, {\tt anton.arnold@tuwien.ac.at, gayrat.toshpulatov@tuwien.ac.at.}}}

\maketitle

\pagestyle{plain}

\begin{abstract}
We consider the relativistic, spatially inhomogeneous Fokker-Planck equation with an external confining potential. 
We prove  the exponential time decay  of solutions towards the global equilibrium  in  weighted $L^2$ and Sobolov spaces. Our result holds for a wide class of external potentials and the  estimates on the rate of convergence are explicit and constructive. Moreover, we prove that the associated semigroup of the equation has hypoelliptic regularizing
properties and we obtain explicit rates on this regularization. The technique is based on the construction of suitable Lyapunov functionals.
\end{abstract}
\tableofcontents
\section{Introduction}
In this paper, we study the long-time behavior of the relativistic, spatially inhomogeneous Fokker-Planck equation \cite{1+3, AC}
\begin{equation}\label{Eq0}
\begin{cases}
\partial_t f+\displaystyle \frac{p}{m\sqrt{1+\frac{|p|^2}{m^{2}c^{2}}}}\cdot \nabla_x f-q\nabla_x V(x) \cdot \nabla_pf=\text{div}_{p}(\sigma D(p)\nabla_p f+\nu pf), \, \, \, \, \, \, x,\, p\in \mathbb{R}^d, \, \, t>0 \\
f_{|t=0}=f_0
\end{cases},
\end{equation}
for any $d\in \mathbb{N}.$ This  kinetic model describes the time evolution of a system with a large number of particles (e.g. in a plasma) undergoing diffusion and friction. In \eqref{Eq0}, $x$ denotes position, $p$ the relativistic momentum,  and $\displaystyle \frac{p}{m\sqrt{1+\frac{|p|^2}{m^{2}c^{2}}}}$  the velocity.  The  unknown  $f=f(t,x,p)\geq 0$ represents the evolution of the phase space probability density of particles. The left hand side is the transport operator with force field $-\nabla_x V(x)$, while the right
hand side describes the diffusion of particles and the interaction with the environment. The (positive) physical constants denoted by $m,$ $c,$  $q,$  $\sigma,$ and $\nu$ are, respectively, the particle rest mass, the vacuum speed of light, the particle charge, diffusion, and  friction coefficients.
$D(p) $ is the relativistic diffusion matrix given by
\begin{equation}\label{matrixD}
  D=\frac{I+\frac{p\otimes p}{m^2c^2}}{\sqrt{1+\frac{|p|^2}{m^2c^2}}}\in \mathbb{R}^{d\times d},    
\end{equation}
where $I\in \mathbb{R}^{d\times d}$ is the identity matrix and $\otimes$ denotes the Kronecker product.  
For a derivation of \eqref{Eq0} and the matrix \eqref{matrixD} from relativistic Langevin dynamics, we refer to \cite{1+3}. Moreover, the same model is obtained in \cite[\S2]{AC}  under the postulate that it should be Lorentz invariant in the absence of friction, i.e.\ for $\nu=0$ (just like the non-relativistic analog is Galilean invariant for $\nu=0$): 
The resulting anisotropic diffusion arises as the Laplace-Beltrami operator over the relativistic hyperboloid with respect to the Minkowski metric. 
Also, this model is compatible with the finite propagation speed of particles in relativity. 
Finally, we remark that there are several different ``relativistic Fokker-Planck equations'' in the literature, for a review see \cite{DH2009}.

 Equation \eqref{Eq0} has several properties following standard physical considerations.
Whenever $f(t,x,p)$ is a (well-behaved) solution of \eqref{Eq0},
one has  \textit{ global conservation of mass}
\begin{equation}\label{cons.mass}   \displaystyle \int_{\mathbb{R}^{2d}}f(t,x,p)dxdp=\int_{\mathbb{R}^{2d}}f_0(x,p)dxdp, \, \, \, \, \, \, \forall t\geq 0.
\end{equation} 
 Therefore, without loss of generality, we shall assume 
$f_0\geq 0$ and $ \displaystyle \int_{\mathbb{R}^{2d}}f_0(x,p)dxdp=1.$\\
 If $V$ grows fast enough,
 \eqref{Eq0} has a unique normalized \textit{steady state} or \textit{global equilibrium} \cite[Section 3.4]{AC}  given by 
\begin{equation}\label{steady state}
f_{\infty}(x,p)= 
\rho_{\infty}(x)M(p),
\end{equation} 
where \begin{equation}\label{star1}\rho_{\infty}(x)\colonequals \displaystyle \frac{e^{-\frac{mq\nu}{\sigma}V(x)}}{ \int_{\mathbb{R}^{d}}e^{-\frac{qm\nu}{\sigma}V(x')}dx' },\, \, \, \, \, \, \, \displaystyle M(p)\colonequals \frac{e^{-\frac{mc\nu}{\sigma}\sqrt{m^2c^2+|p|^2}}}{ \int_{\mathbb{R}^{d}}e^{-\frac{mc\nu}{\sigma}\sqrt{m^2c^2+|p'|^2}}dp'}.
\end{equation} 
 $M(p)$  is called the Maxwell-J\"uttner
distribution, for its normalizing integral see \cite{zaninetti2020}.

\eqref{Eq0} is \textit{dissipative} in the sense that the relative entropy or free energy functional decreases: let $H$ be a functional defined on the space of probability densities  by
\begin{equation*}
f \mapsto H[f]\colonequals \int_{\mathbb{R}^{2d}} f \ln{\frac{f}{f_{\infty}}}dxdp
\end{equation*} 
($f$ is not necessarily the solution). We note that $H[f_{\infty}]=0$ and  $\displaystyle H[f]\geq \frac{1}{2}||f-f_{\infty}||^2_{L^1(\mathbb{R}^{2d})}$ by the Csisz\'ar-Kullback-Pinsker inequality \cite{Cs}.  Hence,  the minimum of $H$ is zero and  it is attained at $f_{\infty}.$   If $f=f(t,x,p)$ is a smooth solution of   \eqref{Eq0}, we have (see \cite[Section 3.3]{AC}) $$\displaystyle \frac{d}{dt}H[f(t)] 
 \leq 0. $$
 This decay of the functional $H$ is a version of Boltzmann's  $H-$theorem stated for the Boltzmann equation \cite{Cercig, Vil.H}. 

 On the basis of the decay of the functional $H,$  one may conjecture that  $H[f(t)]$ decreases to its minimum (which is zero) as $t\to \infty.$ Since this minimum   is obtained at $f_{\infty},$ one can expect that $f(t)$ converges to the equilibrium distribution $f_{\infty}$ as $t\to \infty.$  
 We shall  therefore tackle the interesting  problem to prove (or disprove) that solutions of \eqref{Eq0} converge towards this
equilibrium as $t\to \infty$ and to estimate the convergence rate with constructive bounds. Such explicit and constructive estimates are essential for applications  in physics (equilibration process, numerical simulations). 

 Equation \eqref{Eq0} was introduced in \cite[Eq.(47)]{1+3} and \cite[Eq.(8)]{AC} as a relativistic generalization of the classical kinetic Fokker-Planck  equation \cite{Chan, Chan1, Risken}
\begin{equation}\label{KFP}
\begin{cases}
\partial_t f+\displaystyle \frac{p}{m}\cdot \nabla_x f-q\nabla_x V \cdot \nabla_pf=\text{div}_{p}(\sigma \nabla_p f+\nu pf), \, \, \, \, \, \, x,\, p\in \mathbb{R}^d, \, \, t>0 \\
f_{|t=0}=f_0
\end{cases}.
\end{equation}
This classical equation can be obtained from  \eqref{Eq0} by formally taking the Newtonian limit $c\to \infty.$ In \cite{Felix.Cal.},  this formal limit was justified in the sense that  solutions of \eqref{Eq0} converge  to solutions of \eqref{KFP} in $L^1$ as $c\to \infty.$
Equation \eqref{KFP} is inconsistent with  relativistic mechanics because  it has  infinite speed of propagation: if the particles are initially in a compact region (i.e. $f_0(x,p)$ has  compact support with respect to $x$ and $p$), then, after any short time $t>0,$  we can find particles everywhere with non-zero probability (i.e. $f(t,x,p)>0$), see \cite[Appendix A.22]{Vil.M}. This property   contradicts the law of special relativity that particles can not move faster than light.  By contrast, Equation \eqref{Eq0} is compatible with this physical law as it  exhibits  finite speed of propagation  w.r.t. the $x$ variable \cite[Section 3.2]{AC}. Note, however, that the degenerate parabolicity of \eqref{Eq0}
does entail infinite speed of propagation with respect to the $p$ variable.

The classical equation \eqref{KFP} has been studied  comprehensively: well-posedness and hypoelliptic regularity were obtained in  \cite{NH, Her.el., Vil.M}. The  long-time behavior of \eqref{KFP} was studied  in \cite{H.N} for fast growing potentials. By using hypocoercive methods, Villani proved exponential convergence results in \cite{DesVil, Vil.M}. This result was extended  in \cite{Baudoin} for potentials with singularities. In \cite{dolbeault2015hypocoercivity}, Dolbeault, Mouhot, and Schmeiser developed a method to
obtain exponential decay in $L^2$ for a large class of linear kinetic equations, and, as an application, an exponential decay in $L^2$ was proven for \eqref{KFP}. Their method
was also used to study the long-time behavior of \eqref{KFP} when the potential $V$ is zero or grows slowly as
$|x| \to \infty,$ see \cite{confint, 12}.  Based on a probabilistic coupling method, Eberle, Guillin, and Zimmer \cite{Wass}
obtained an exponential decay result in Wasserstein distance. We also refer the recent work \cite{AT}  where sharper exponential rates were obtained using a modified entropy method.  

 Concerning the relativistic equation \eqref{Eq0}, there are only few studies: global existence and uniqueness  were proven in \cite{AC}.
 The long-time behavior of spatially homogeneous solutions of \eqref{Eq0} was studied \cite{Angst, Felix.Cal.}, where the authors used logarithmic Sobolev inequalities and entropy methods \cite{Con.Sob, Ansgar.J}. When  \eqref{Eq0} is supplemented with periodic boundary conditions (i.e. $x\in \mathbb{T}^d$) and $V=0,$  exponential decay of solutions to the steady state was proven in  \cite{Cal} by using the hypocoercive method developed by Villani \cite{Vil.M}.

In this paper,
we shall improve these previous results when there is a non-zero potential
$V.$ For the full system \eqref{Eq0} with a non-zero potential $V$ we
shall prove the exponential convergence $f (t) \to  f_{\infty}$ as $t \to \infty$ for a wide class of
potentials $V.$ Our convergence rates  are explicit and constructive. We show that, although the equation is only degenerate parabolic, the equation has instantaneous regularizing
properties which is called \textit{hypoellipticity} \cite{horm}.   We provide  explicit rates on this regularization. 
  We believe our results are the first convergence and regularity results for \eqref{Eq0} with a non-zero potentials $V.$

The organization of this paper is as follows. In Section \ref{sec:main}, we define the assumptions on the
potential and state the main results. Sections \ref{sec:L2-conv} and \ref{sec:H1-conv} are devoted to prove the convergence $f (t) \to  f_{\infty}$ as $t \to \infty$ in the weighted $L^2$  and Sobolev spaces. We study hypoelliptic regularity properties of the equation in Section \ref{sec:hypoelliptic} and list several technical results in the Appendix.  
In Section \ref{sec:limit} we show that the decay rate for the homogeneous relativistic Fokker-Planck equation is uniform in the limit as $c\to\infty$, and we conclude in Section \ref{sec:conclusion}. 

\section{Setting and main result}\label{sec:main}
We use the notations $$V_0(x)\colonequals \sqrt{1+|\nabla_x V(x)|^2}\, \, \, \text{ and }\, \, \, p_0(p)\colonequals \sqrt{1+|p|^2}.$$
For simplicity, we set all physical constants to unity: $m=c=q=\sigma=\nu=1.$  Therefore, we shall consider the normalized equation
\begin{equation}\label{Eq}
\begin{cases}
\partial_t f+\displaystyle \frac{p}{p_0}\cdot \nabla_x f-\nabla_x V(x) \cdot \nabla_pf=\text{div}_{p}( D(p)\nabla_p f+ pf), \, \, \, \, \, \, x,\, p\in \mathbb{R}^d, \, \, t>0 \\
f_{|t=0}=f_0
\end{cases}
\end{equation}
with $\displaystyle D(p)=\frac{I+p\otimes p}{p_0}.$
We define the  weighted spaces $L^2(\mathbb{R}_x^{d}, \rho_{\infty}),$ $L^2(\mathbb{R}_p^{d}, M),$ and $L^2(\mathbb{R}^{2d}, f_{\infty})$   as the Lebesgue spaces associated, respectively, to the norms
$$||g||_{L^2(\mathbb{R}_x^{d}, \rho_{\infty})}\colonequals \sqrt{\int_{\mathbb{R}^{d}}g^2\rho_{\infty}dx}, \, \, \, ||g||_{L^2(\mathbb{R}_p^{d}, M)}\colonequals \sqrt{\int_{\mathbb{R}^{d}}g^2Mdp},$$
and $$||g||_{L^2(\mathbb{R}^{2d}, f_{\infty})}\colonequals \sqrt{\int_{\mathbb{R}^{2d}}g^2f_{\infty}dxdp}. $$
We note that $\displaystyle M(p)= \frac{e^{-\sqrt{1+|p|^2}}}{ \int_{\mathbb{R}^{d}}e^{-\sqrt{1+|p'|^2}}dp'}$ gives rise to the following Poincar\'e inequality 
\cite[Theorem 3]{Felix.Cal.}: there is a positive  constant $\kappa_1$ such that 
\begin{equation}\label{Poincare M}
\int_{\mathbb{R}^{d}} h^2 Mdp-\left(\int_{\mathbb{R}^{d}} h M dp \right)^2\leq \frac{1}{\kappa_1}\int_{\mathbb{R}^{d}} \nabla^{T}_p h D\nabla_p h  Mdp
\end{equation} holds for all $h\in L^2(\mathbb{R}_p^{d}, M) $ with $\int_{\mathbb{R}^{d}} \nabla^{T}_p h D\nabla_p h  Mdp<\infty. $

We shall assume that $\displaystyle \rho_{\infty}(x)= \frac{e^{-V(x)}}{ \int_{\mathbb{R}^{d}}e^{-V(x')}dx' }$ also gives rise to a Poincar\'e inequality. Also we shall assume  some growth conditions on $V:$
\begin{assumption}\label{Assumptions i-ii} \begin{itemize}
\item[i)] 
Let  $V\in C^{2}(\mathbb{R}^d)$ be such that $e^{-V} \in L^1(\mathbb{R}^d),$ and   there exists a constant $\kappa_2>0$ such that the Poincar\'e inequality 
    \begin{equation}\label{Poincare rho}
  \int_{\mathbb{R}^{d}} h^2 \rho_{\infty}dx-\left(\int_{\mathbb{R}^{d}} h \rho_{\infty}dx\right)^2\leq \frac{1}{\kappa_2}\int_{\mathbb{R}^{d}} |\nabla_x h|^2 \rho_{\infty}dx\, \, 
\end{equation}
holds for all $h\in L^2(\mathbb{R}_x^{d}, \rho_{\infty}) $ with $|\nabla_x h|\in L^2(\mathbb{R}_x^{d}, \rho_{\infty}). $
\item[ii)] There exist constants $c_1>0,$  $c_2\in [0,1),$ and $c_3>0$ such that  \begin{equation}\label{cond.V}
\Delta_x V(x) \leq c_1+\frac{c_2}{2}|\nabla_x V(x)|^2, \, \, \, \, \, \, \left|\left|\frac{\partial^2 V(x)}{\partial x^2}\right|\right|_F\leq c_3(1+|\nabla_x V(x)|), \, \, \, \, \forall\, x\in \mathbb{R}^d,
\end{equation}
where $\displaystyle \left|\left|\frac{\partial^2 V(x)}{\partial x^2}\right|\right|_F\colonequals \sqrt{\sum_{i,j=1}^d (\partial_{x_i x_j}V(x))^2}$ is the Frobenius norm of $\displaystyle \frac{\partial^2 V(x)}{\partial x^2}.$
\end{itemize}
\end{assumption}
There are a lot of studies and sufficient conditions implying the Poincar\'e inequality \eqref{Poincare rho}. For example, if $V$ is uniformly convex (Bakry-Emery criterion) or if $$ \liminf_{|x|\to \infty} \left(a|\nabla V(x)|^2-\Delta V(x)\right)> 0
$$  
 for some $a \in (0,1),$ then the Poincar\'e inequality \eqref{Poincare rho} holds.  For more information  see \cite{BBCG}, \cite[Chapter 4]{BGL}.
We note that the potentials of the form $$V(x)=r|x|^{2k}+\tilde{V}(x),$$
where $r>0,$ $k>1$ and $\tilde{V}\colon \mathbb{R}^d \to \mathbb{R}$ is a polynomial of degree  $j<2k,$ satisfy our assumptions.

We now state our first result:
\begin{theorem}[\textbf{Exponential decay in $L^2(\mathbb{R}^{2d}, f_{\infty})$}]\label{main result}
Let $\frac{f_0}{f_{\infty}}\in L^2(\mathbb{R}^{2d}, f_{\infty})$ and $V$ satisfy Assumption \ref{Assumptions i-ii}.  Then there are explicitly computable constants $C_1>0$ and $\lambda>0$ (independent of $f_0$) such that
\begin{equation}\label{L2-decay}
\left|\left|\frac{f(t)-f_{\infty}}{f_{\infty}}\right|\right|_{L^2(\mathbb{R}^{2d}, f_{\infty})}\leq C_1 e^{-\lambda t} \left|\left|\frac{f_0-f_{\infty}}{f_{\infty}}\right|\right|_{L^2(\mathbb{R}^{2d}, f_{\infty})}
\end{equation}
holds for all $t\geq 0.$   
\end{theorem}

Theorem \ref{main result} shows that the solution $ \frac{f(t)-f_{\infty}}{f_{\infty}}$ converges  exponentially to zero in $L^2(\mathbb{R}^{2d}, f_{\infty})$  as $t\to \infty.$  Next we want to obtain this convergence in a more regular space. Hence, 
we define the following weighted Sobolev space
$\mathscr{H}^1(\mathbb{R}^{2d}, f_{\infty})$ associated to the norm
\begin{align}\label{H^1-norm}
||h||^2_{\mathscr{H}^1(\mathbb{R}^{2d}, f_{\infty})} &\colonequals \int_{ \mathbb{R}^{2d}}h^2f_{\infty}dxdp +\int_{\mathbb{R}^{2d}}\frac{1}{V_0^3(x)p_0^3}\nabla^T_x h\left(I-\frac{p\otimes p}{p_0^2}\right)\nabla_x hf_{\infty}dxdp \nonumber \\ 
&\, \, \, \, \,\, \, \,  +\int_{\mathbb{R}^{2d}}\frac{1}{V_0(x)p_0}\nabla^T_p h(I+p \otimes p)\nabla_p hf_{\infty}dxdp \nonumber \\ 
&\, \,=\int_{\mathbb{R}^{2d}}h^2f_{\infty}dxdp+\int_{ \mathbb{R}^{2d}}\frac{1}{V_0^3(x)p_0^{3}}\left(|\nabla_x h|^2-\frac{|p\cdot \nabla_x h|^2}{p_0^2}\right)f_{\infty}dxdp \nonumber \\ 
&\, \,\, \, \, \, \, \,  +\int_{ \mathbb{R}^{2d}}\frac{1}{V_0(x)p_0}(|\nabla_p h|^2+|p\cdot \nabla_p h|^2)f_{\infty}dxdp.
\end{align}
This norm is well-defined since the matrices $\displaystyle \frac{1}{V_0^3(x)p_0^3}\left(I-\frac{p\otimes p}{p_0^2}\right)$ and $\displaystyle \frac{1}{V_0(x)p_0}(I+p \otimes p)$ are positive definite for all $x, \,p\in \mathbb{R}^d.$ Clearly, $\mathscr{H}^1(\mathbb{R}^{2d}, f_{\infty})\subset L^2(\mathbb{R}^{2d}, f_{\infty}).$ 
To motivate the above definition better, we note that $I-\frac{p\otimes p}{p_0^2} = \frac{1}{p_0}D(p)^{-1}$.

 Our second result shows that the solution $ \frac{f(t)-f_{\infty}}{f_{\infty}}$ converges exponentially to zero in $\mathscr{H}^1(\mathbb{R}^{2d}, f_{\infty})$   as $t\to \infty:$  
\begin{theorem}[\textbf{Exponential decay in $\mathscr{H}^1(\mathbb{R}^{2d}, f_{\infty})$}]\label{decay in H^1}
Let $\frac{f_0}{f_{\infty}}\in \mathscr{H}^1(\mathbb{R}^{2d}, f_{\infty})$ and $V$ satisfy Assumption \ref{Assumptions i-ii}.
Then there are constants $C_2>0$ and $\Lambda>0$ (independent of $f_0$) such that 
\begin{equation*}
\left|\left|\frac{f(t)-f_{\infty}}{f_{\infty}}\right|\right|_{\mathscr{H}^1(\mathbb{R}^{2d}, f_{\infty})}\leq C_2 e^{-\Lambda t} \left|\left|\frac{f_0-f_{\infty}}{f_{\infty}}\right|\right|_{\mathscr{H}^1(\mathbb{R}^{2d}, f_{\infty})}
\end{equation*}
holds for all $t\geq 0.$   
\end{theorem}
Our next result is about the estimates on the hypoelliptic regularization:
\begin{theorem}[\textbf{Hypoelliptic regularity from $L^2(\mathbb{R}^{2d}, f_{\infty})$ to  $\mathscr{H}^1(\mathbb{R}^{2d}, f_{\infty})$}]\label{th:hypoellipticity in H^1}
Assume $\frac{f_0}{f_{\infty}}\in L^2(\mathbb{R}^{2d}, f_{\infty})$ and   that there exists a constant $c_3>0$ such that  \begin{equation*}
 \left|\left|\frac{\partial^2 V(x)}{\partial x^2}\right|\right|_F\leq c_3(1+|\nabla_x V(x)|), \, \, \, \, \forall\, x\in \mathbb{R}^d.
\end{equation*} Then, for any $t_0>0,$  there are explicitly computable constants $C_3>0$ and $C_4>0$ (independent of $f_0$)  such that
\begin{equation}\label{hypel11}
\int_{ \mathbb{R}^{2d}}\frac{1}{V_0^3(x)p_0^3}\nabla^T_x \left(\frac{f(t)}{f_{\infty}}\right)\left(I-\frac{p\otimes p}{p_0^2}\right)\nabla_x \left(\frac{f(t)}{f_{\infty}}\right)f_{\infty}dxdp\leq \frac{C_3}{t^3}\int_{\mathbb{R}^{2d}}\left(\frac{f_0-f_{\infty}}{f_{\infty}}\right)^2f_{\infty}dxdp
\end{equation} and
\begin{equation}\label{hypoel21}
\int_{ \mathbb{R}^{2d}}\frac{1}{V_0(x)p_0}\nabla^T_p \left(\frac{f(t)}{f_{\infty}}\right)\left(I+p\otimes p\right)\nabla_p \left(\frac{f(t)}{f_{\infty}}\right)f_{\infty}dxdp\leq \frac{C_4}{t}\int_{ \mathbb{R}^{2d}}\left(\frac{f_0-f_{\infty}}{f_{\infty}}\right)^2f_{\infty}dxdp 
\end{equation} 
hold for all $t \in (0,t_0].$ In particular,  
\begin{equation}\label{hypoel3}
\left|\left|\frac{f(t)-f_{\infty}}{f_{\infty}}\right|\right|_{\mathscr{H}^1(\mathbb{R}^{2d}, f_{\infty})}\leq \frac{(C_3+C_4 t_0^2)^{1/2}}{t^{3/2}}  \left|\left|\frac{f_0-f_{\infty}}{f_{\infty}}\right|\right|_{L^2(\mathbb{R}^{2d}, f_{\infty})}
\end{equation}
holds for all $t \in (0,t_0].$
\end{theorem}
Theorem \ref{th:hypoellipticity in H^1} shows that, for any initial data $\frac{f_0}{f_{\infty}}\in L^2(\mathbb{R}^{2d}, f_{\infty}),$ the solution $\frac{f(t)}{f_{\infty}}\in \mathscr{H}^1(\mathbb{R}^{2d}, f_{\infty})$ for any time $t>0.$ Compared to Theorem \ref{main result} and Theorem \ref{decay in H^1}, we  do not require the validity of a  Poincar\'e inequality in Theorem \ref{th:hypoellipticity in H^1}.  Also  note that the regularization rates for the $x$ derivative and the $p$ derivative are not the same: the regularization rate  in the $p$ derivative is faster, as it also is for the kinetic Fokker-Planck equation \cite{Her.el., Vil.M, AT}. This difference is expected since \eqref{Eq} can be considered as a transport equation with respect to the $x$ variable and as a parabolic equation  with respect to the $p$ variable.

In Theorem \ref{decay in H^1} we assumed that the initial data $f_0 /f_{\infty}$ is in $\mathscr{H}^1(\mathbb{R}^{2d}, f_{\infty}).$ If we use the regularity estimates
from Theorem \ref{th:hypoellipticity in H^1}, this condition can be relaxed:
\begin{corollary}\label{corollary}
Assume $\frac{f_0}{f_{\infty}}\in L^2(\mathbb{R}^{2d}, f_{\infty})$ and $V$ satisfies Assumption \ref{Assumptions i-ii}. Then, for any $t_0>0,$  there is an explicitly computable constant $C_5>0$  (independent of $f_0$)  such that
\begin{equation*}
\left|\left|\frac{f(t)-f_{\infty}}{f_{\infty}}\right|\right|_{\mathscr{H}^1(\mathbb{R}^{2d}, f_{\infty})}\leq C_5 e^{-\Lambda t} \left|\left|\frac{f_0-f_{\infty}}{f_{\infty}}\right|\right|_{L^2(\mathbb{R}^{2d}, f_{\infty})}
\end{equation*}
holds for all $t\geq t_0>0,$ where $\Lambda>0$ is the constant appearing in Theorem \ref{decay in H^1}.   
\end{corollary}
  
\begin{Remark}
If one  considers  \eqref{Eq}  on a torus as done in \cite{Cal}, our results also hold in this setting since the method which we use can  be adapted without  difficulty.

\end{Remark} 
\section{Exponential convergence in $L^2(\mathbb{R}^{2d}, f_{\infty})$}\label{sec:L2-conv}
\subsection{The first Lyapunov functional}

Let us consider the relativistic, spatially homogeneous Fokker-Planck equation
\begin{equation}\label{HFP}
\begin{cases}
\partial_t \varrho=\text{div}_{p}(D(p)\nabla_{p}\varrho+p\varrho), \, \, \, \, \, p \in \mathbb{R}^d, \, \, \, \, t>0,\\
\varrho_{|t=0}=\varrho_0,
\end{cases}
\end{equation}
with $\int_{\mathbb{R}^d} \varrho dp=1.$
This equation is a special case of \eqref{Eq} when the problem is independent of $x$ and $V=0.$ The unique normalized global equilibrium for this equation is $\displaystyle M(p)=\frac{e^{-\sqrt{1+|p|^2}}}{ \int_{\mathbb{R}^{d}}e^{-\sqrt{1+|p'|^2}}dp'}.$ The convergence $$\varrho(t)\to M\, \, \, \text{ as }\, \, t\to \infty$$ can be easily proven using the Poincar\'e inequality \eqref{Poincare M}: Noting that the right hand side of \eqref{HFP} equals $\displaystyle \text{div}_p\left(MD\nabla_p\left(\frac{\varrho}{M}\right)\right)$ we obtain
\begin{align*} 
\frac{d}{dt} \left|\left|\frac{\varrho(t)-M}{M}\right|\right|^2_{L^2(\mathbb{R}^d, M)}&=-2\int_{\mathbb{R}^d}\nabla^T_{p}\left(\frac{\varrho(t)}{M}\right)D \nabla_{p}\left(\frac{\varrho(t)}{M}\right) M dp\\ &\leq -2\kappa_1 \left|\left|\frac{\varrho(t)-M}{M}\right|\right|^2_{L^2(\mathbb{R}^d, M)},\, \, \, \, \, \, \, \, \, \, \, \, \, \, \, \, \, \, \, \, \, \, \, \, \, \, \, \, \, \,  \forall \,t>0.
\end{align*}
 By Gr\"onwall's lemma we obtain the exponential decay  
 \begin{equation}\label{p-decay}
 \left|\left|\frac{\varrho(t)-M}{M}\right|\right|_{L^2(\mathbb{R}^d, M)}\leq e^{-\kappa_1 t}\left|\left|\frac{\varrho_0-M}{M}\right|\right|_{L^2(\mathbb{R}^d, M)},  \, \, \, \, \, \, \, \, \, \forall\, t>0.
 \end{equation}

  On the contrary, we do not obtain easily such exponential decay for the relativistic, spatially  inhomogeneous Fokker-Planck equation \eqref{Eq}. As the Fokker-Planck operator on the right hand side of \eqref{Eq} acts only on the variable $p,$ we only have
     \begin{equation}\label{dtint}
\frac{d}{dt} \left|\left|\frac{f(t)-f_{\infty}}{f_{\infty}}\right|\right|^2_{L^2(\mathbb{R}^{2d},f_{\infty})}=-2\int_{\mathbb{R}^{2d}}\nabla^T_{p}\left(\frac{f(t)}{f_{\infty}}\right)D \nabla_{p}\left(\frac{f(t)}{M}\right) f_{\infty}dxdp\leq 0.
\end{equation}
The integral on the right hand side only gives information on the $p-$derivative and it is lacking information on the $x-$derivatives. Hence,  in general, the integral on the right hand side of \eqref{dtint} is not  bigger than  \begin{small}$\displaystyle 2\lambda  \left|\left|\frac{f(t)-f_{\infty}}{f_{\infty}}\right|\right|^2_{L^2(\mathbb{R}^{2d},f_{\infty})}$ \end{small}for some $\lambda>0.$

 The idea to overcome this difficulty  is to construct an appropriate {Lyapunov functional}   which is equivalent to the $L^2-$norm and satisfies a Gr\"onwall type differential inequality under the evolution of the solution $f(t).$
A method in Hilbert spaces was introduced by Dolbeault, Mouhot, and Schmeiser in \cite{dolbeault2009hypocoercivity, dolbeault2015hypocoercivity} for proving exponential stability for a large class of linear kinetic models confined by an external potential. We will apply this method to \eqref{Eq}  to obtain our results. In the following we explain this method for linear kinetic equations of the form
\begin{equation}\label{abs.kin.eq}
\partial_t f+\mathrm{T}f=\mathrm{L}f, \,\,\, t>0
\end{equation} 
in a Hilbert space $\mathcal{H}$ with initial data $f_{|t=0}=f_0\in \mathcal{H}. $  Here, $\mathrm{T}$ and $\mathrm{L}$ are closed linear operators such that $\mathrm{L}-\mathrm{T}$ generates the strongly continuous semigroup $e^{(\mathrm{L}-\mathrm{T})t}$ on $\mathcal{H}.$ Let $\mathrm{I}$ be the identity operator, $\Pi$ be the orthogonal projection on the null space $\mathcal{N}(\mathrm{L})$ of $\mathrm{L}.$ The domains of $\mathrm{T}$ and $\mathrm{L}$ are denoted by $\mathcal{D}(\mathrm{T})$ and $\mathcal{D}(\mathrm{L}),$ respectively. We define the operator 
\begin{equation}\label{A oper.}
\mathrm{A}f\colonequals (\mathrm{I}+(\mathrm{T}\Pi)^{*}\mathrm{T}\Pi)^{-1}(\mathrm{T}\Pi)^{*}f
\end{equation}
 and the functional 
\begin{equation}\label{DMS funct}
\mathrm{H}_{\delta}[f]\colonequals \frac{1}{2}||f||^2+\delta \langle \mathrm{A}f,f\rangle,  \, \, \, \, \delta>0,
\end{equation}
where $\langle \cdot, \cdot \rangle$ denotes the scalar product in $\mathcal{H},$ and  $||\cdot||$ denotes the norm on $\mathcal{H}$ associated with the scalar product. 
 We assume the following conditions are satisfied:
\begin{itemize}
\item ({\it microscopic coercivity}) $\mathrm{L}$ is symmetric and there exists $\lambda_m>0$ such that 
\begin{equation}\label{mic.coer}
-\langle \mathrm{L}f, f\rangle \geq \lambda_m ||(\mathrm{I}-\Pi)f||^2 \, \, \, \, \text{ for all } \,\, f \in  \mathcal{D}(\mathrm{L}).
\end{equation}
\item ({\it macroscopic coercivity}) $\mathrm{T}$ is skew symmetric and there exists $\lambda_M>0$ such that 
\begin{equation}\label{mac.coer}
||\mathrm{T} \Pi f||^2 \geq \lambda_M ||\Pi f||^2 \, \, \, \, \text{ for all } \,\, f \in \mathcal{H}\,\, \text{ with }\, \, \Pi f \in  \mathcal{D}(\mathrm{T}).
\end{equation}
\item \begin{equation}\label{pr.tr.pr}
\Pi \mathrm{T}\Pi=0.
\end{equation}
\item (\textit{boundedness of auxiliary operators}) The operators $\mathrm{A}\mathrm{T}(\mathrm{I}-\Pi)$ and $\mathrm{A} \mathrm{L}$ are bounded, and there exists a constant $C_M>0$ such that, for all $f \in \mathcal{H},$
\begin{equation}\label{boun.aux}
||\mathrm{A}\mathrm{T}(\mathrm{I}-\Pi)f||+||\mathrm{A} \mathrm{L}f||\leq C_M ||(\mathrm{I}-\Pi)f||.
\end{equation}
\end{itemize}
We define \begin{equation}\label{delta0}\delta_0\colonequals \min\left\{2, \lambda_m, \frac{4\lambda_m \lambda_M}{4\lambda_M+C^2_M(1+\lambda_M)}\right\}. 
\end{equation}
 Under the validity of these conditions and for $\delta\in (0,\delta_0),$ one can show that $\mathrm{H}_{\delta}$  is a {Lypunov functional} for \eqref{abs.kin.eq} and it decays exponentially:       
\begin{theorem}[{\cite[Theorem 2]{dolbeault2015hypocoercivity}}]\label{DMS}
Assume  \eqref{mic.coer}-\eqref{boun.aux} are satisfied and $\delta\in (0,\delta_0).$ Then,
\begin{enumerate}[~~i)]
\item  $\mathrm{H}_{\delta}$ and $||\cdot||^2$ are equivalent, more precisely,
\begin{equation*}
\frac{2-\delta}{4}||f||^2\leq \mathrm{H}_{\delta}[f] \leq \frac{2+\delta}{4}||f||^2 \, \, \, \, \text{ for all }  f \in \mathcal{H}.
\end{equation*} 
\item   There exists a positive constant $\lambda,$ which is computable in terms of $\lambda_m,$ $\lambda_M,$ and $C_M,$ such that, for any initial data $f_0 \in \mathcal{H},$ 
\begin{equation*}
\frac{d}{dt}\mathrm{H}_{\delta}[e^{(\mathrm{L}-\mathrm{T})t}f_0]\leq -2\lambda \mathrm{H}_{\delta}[e^{(\mathrm{L}-\mathrm{T})t}f_0], \, \, \, \, t>0. 
\end{equation*} 
Consequently, we have 
\begin{equation}\label{|| decay}
||e^{(\mathrm{L}-\mathrm{T})t}f_0||\leq \sqrt{\frac{2+\delta}{2-\delta}}e^{-\lambda t}||f_0||\, \, \, \, \text{ for all } t\geq 0.
\end{equation}
\end{enumerate}
\end{theorem}
This method has been successfully applied to study the long-time behavior of various linear kinetic models,  see \cite{dolbeault2015hypocoercivity, confint, Schmeis, VPFP}. In particular, in \cite[Theorem 10]{dolbeault2015hypocoercivity}, the exponential convergence $f(t)\to f_{\infty}$ in $L^2(\mathbb{R}^{2d}, f_{\infty})$ as $t\to \infty$ was proven for the classical kinetic Fokker-Planck equation \eqref{KFP}.
\subsection{Proof of Theorem \ref{main result}}
Let $\displaystyle h\colonequals \frac{f-f_{\infty}}{f_{\infty}}.$ Then \eqref{Eq} can be written as \begin{equation}\label{Eqh}
\begin{cases}
\partial_t h+\displaystyle  \frac{p}{p_0}\cdot \nabla_x h-\nabla_x V \cdot \nabla_p h=\frac{1}{f_{\infty}}\text{div}_{p}(D\nabla_p h f_{\infty})\\
\displaystyle h_{|t=0}=\frac{f_0-f_{\infty}}{f_{\infty}}.
\end{cases}
\end{equation}
We shall apply  Theorem \ref{DMS} to \eqref{Eqh}. To end we first  define an appropriate Hilbert space,
\begin{equation*}
\mathcal{H}\colonequals \left\{ h \in  L^2(\mathbb{R}^{2d},f_{\infty}):\int_{\mathbb{R}^{2d}}h f_{\infty} dxdp=0  \right\}
\end{equation*} with the scalar product
$
\langle h,g \rangle\colonequals \int_{\mathbb{R}^{2d}}{h_1h_2}{f_{\infty}}dxdp$ and the norm $  ||h||_{L^2(\mathbb{R}^{2d}, f_{\infty})}= \sqrt{\langle h,h \rangle}.$
We note that if $h(t)$ is the solution of \eqref{Eqh}, then the conservation of mass \eqref{cons.mass} shows $h(t)\in \mathcal{H}$ for all $t\geq  0.$
We can present \eqref{Eqh} in the form of \eqref{abs.kin.eq} with
\begin{equation}\label{T}
\mathrm{T}h\colonequals \frac{p}{p_0}\cdot\nabla_x h-\nabla_x V \cdot \nabla_p h
\end{equation}
and
\begin{equation}\label{L}
\mathrm{L}h\colonequals \frac{1}{f_{\infty}}\text{div}_{p}(D\nabla_p h f_{\infty}).
\end{equation} 
$\mathrm{T}$ and   $\mathrm{L}$ can be defined  in the space of smooth functions with  compact support. These operators can be extended using the Friederichs extension, but we omit details concerning domain issues and extensions as we need only properties that apply to solutions of the
evolution problem \eqref{Eqh}.

 We define 
\begin{equation}\label{P}
\Pi h=\Pi h(x) \colonequals \int_{\mathbb{R}^d} h(x,p')M(p') dp', \, \, \, \, \, \, \, h \in \mathcal{H}.
\end{equation}
It is easy to check that $\Pi$ is a symmetric operator in $\mathcal{H}$ and $\Pi\circ \Pi=\Pi.$

In the following proposition we show that the operators defined in \eqref{T}-\eqref{P}  satisfy the conditions \eqref{mic.coer}-\eqref{pr.tr.pr}  in $\mathcal{H}:$
\begin{proposition}\label{proporties}
Assume that Assumption \ref{Assumptions i-ii} holds. Then we have 
\begin{enumerate}[~~i)]
\item $\mathrm{T}$ and $\mathrm{L}$ are, respectively, skew-symmetric and symmetric operators in $\mathcal{H}.$
\item  $\Pi$ is  the orthogonal projection on the null space $\mathcal{N}(\mathrm{L})$ of  $\mathrm{L}.$
Microscopic coercivity \eqref{mic.coer} holds with $\lambda_m=\kappa_1,$ where $\kappa_1$ is the constant appearing in the Poincar\'e inequality \eqref{Poincare M}.
\item Macroscopic coercivity \eqref{mac.coer}  holds with $\lambda_M=\kappa_2\left( \int_{\mathbb{R}^d}\frac{1}{(1+|p|^2)^{3/2}}Mdp\right)^{-1},$ where $\kappa_2$ is the constant in the Poincar\'e inequality \eqref{Poincare rho}.
\item $\Pi \mathrm{T} \Pi=0.$
\end{enumerate}
\end{proposition}

\begin{proof}
$i)$ Let  $h,\, g \in  \mathcal{H}$ be smooth functions with compact support.  The equations 
\begin{equation*}\label{grad. f_infty}
\nabla_x f_{\infty}=-\nabla_x V f_{\infty}, \, \, \, \, \, \, \, \, \nabla_p f_{\infty}=-\frac{p}{p_0}f_{\infty},
\end{equation*}
 and integration by parts yield
\begin{align*}
\langle \mathrm{T}h, g\rangle & =\int_{\mathbb{R}^{2d}} \left(\frac{p}{p_0}\cdot\nabla_x h-\nabla_x V \cdot \nabla_p h\right) gf_{\infty}dxdp\\
&=-\int_{\mathbb{R}^{2d}} \left(\frac{p}{p_0}\cdot\nabla_x g-\nabla_x V \cdot \nabla_p g\right) hf_{\infty}dxdp\\ & =-\langle h, \mathrm{T}g\rangle.
\end{align*}
Then, integrating by parts we show that $\mathrm{L}$ is symmetric:
\begin{equation}\label{L(h,g)}
\langle \mathrm{L}h, g\rangle =\int_{\mathbb{R}^{2d}} \text{div}_{p}(D\nabla_p h f_{\infty}) gdxdp
=-\int_{\mathbb{R}^{2d}} \nabla^{T}_p h  D \nabla_p g f_{\infty}dxdp=\langle h, \mathrm{L}g\rangle.
\end{equation}
\\

$ii)$
 As $D=D(p)$ is positive definite for all $p\in \mathbb{R}^d,$ \eqref{L(h,g)} implies
\begin{equation*}
\langle \mathrm{L}h, h\rangle =-\int_{\mathbb{R}^{2d}} \nabla^{T}_p h D \nabla_p h f_{\infty}dxdp\leq 0.
\end{equation*}
This shows that $\mathrm{L}h$ vanishes if  $h$ is constant with respect to $p,$ in particular $\mathrm{L} \Pi h=0.$
 Moreover, the Poincar\'e inequality \eqref{Poincare M} shows
\begin{align*}
-\langle \mathrm{L}h, h\rangle & \geq \kappa_1\int_{\mathbb{R}^{d}}\left(\int_{\mathbb{R}^{d}} h^2 Mdp-\left(\int_{\mathbb{R}^{d}} h M dp \right)^2\right)\rho_{\infty} dx\\
&=\kappa_1\int_{\mathbb{R}^{2d}}(h-\Pi h)^2f_{\infty}dxdp\\
&=\kappa_1||(\mathrm{I}-\Pi)h||^2_{L^2(\mathbb{R}^{2d}, f_{\infty})}.
\end{align*}
This justifies that $\Pi$ is the orthogonal projection on the null space $\mathcal{N}(\mathrm{L})$ of $\mathrm{L}.$\\

$iii)$ Using
$\displaystyle  \mathrm{T}\Pi h=\frac{p}{p_0}\cdot\nabla_x \Pi h$ we compute
\begin{align}\label{TP1}||\mathrm{T}\Pi h||^2_{L^2(\mathbb{R}^{2d}, f_{\infty})}& =\sum_{i,j=1}^{d}\int_{\mathbb{R}^{2d}}\frac{p_i p_j}{p^2_0}\partial_{x_i}\Pi h\, \partial_{x_j}\Pi h f_{\infty}dxdp \nonumber \\& =-\sum_{i,j=1}^{d}\int_{\mathbb{R}^{2d}}\partial_{p_i}{p_0}\partial_{x_i}\Pi h\, \partial_{x_j}\Pi h \partial_{p_j}f_{\infty}dxdp\nonumber \\&=\sum_{i,j=1}^{d}\int_{\mathbb{R}^{2d}}\partial^2_{p_i p_j}{p_0}\partial_{x_i}\Pi h\, \partial_{x_j}\Pi h f_{\infty}dxdp\nonumber \\ &=\int_{\mathbb{R}^{2d}}\nabla^{T}_{x}\Pi h \frac{\partial^2 p_0}{\partial p^2}\nabla_{x}\Pi h f_{\infty}dxdp.
\end{align}
We have \begin{equation}\label{hessian inq.}\frac{\partial^2 p_0}{\partial p^2}=\frac{1}{p_0}\left(I-\frac{p\otimes p}{p_0^2}\right)\geq \frac{1}{p_0^3} I,
\end{equation}
where  $I\in \mathbb{R}^{d\times d}$ is the identity matrix.
  \eqref{TP1} and \eqref{hessian inq.} yield
$$||\mathrm{T}\Pi h||^2_{L^2(\mathbb{R}^{2d}, f_{\infty})}\geq \left(\int_{\mathbb{R}^{d}} \frac{1}{p_0^3}Mdp \right)\int_{\mathbb{R}^{d}}|\nabla_{x}\Pi h|^2\rho_{\infty}dx.$$
Then, the Poincar\'e inequality \eqref{Poincare rho} provides the claimed result. \\

$iv)$ Using $\displaystyle \nabla_p M=-\frac{p}{p_0}M$ and integrating by parts with respect to $p$ we obtain
$$\Pi \mathrm{T} \Pi h=\int_{\mathbb{R}^{d}}\frac{p}{p_0}\cdot\nabla_x \Pi hM dp=-\int_{\mathbb{R}^{d}}\nabla_x \Pi h \cdot \nabla_p M dp=0.$$
\end{proof}

Next we define as in \eqref{A oper.}: 
$\mathrm{A}f\colonequals (\mathrm{I}+(\mathrm{T}\Pi)^{*}\mathrm{T}\Pi)^{-1}(\mathrm{T}\Pi)^{*}f,$ which is a bounded operator on  $\mathcal{H},$ due to \cite[Lemma 1]{dolbeault2015hypocoercivity}.
We now show that  \eqref{boun.aux} holds.
\begin{lemma}\label{LEMMA} Assume Assumption \ref{Assumptions i-ii} holds. Then, the operators $\mathrm{A}\mathrm{T}(\mathrm{I}-\Pi)$ and $\mathrm{A} \mathrm{L}$ are bounded, and there exists a constant $C_M>0$ such that, for all $h \in \mathcal{H},$
\begin{equation*}
||\mathrm{A}\mathrm{T}(\mathrm{I}-\Pi)h||_{L^2(\mathbb{R}^{2d}, f_{\infty})}+||\mathrm{A} \mathrm{L}||_{L^2(\mathbb{R}^{2d}, f_{\infty})}\leq C_M ||(\mathrm{I}-\Pi)h||_{L^2(\mathbb{R}^{2d}, f_{\infty})}.
\end{equation*}
\end{lemma}
\begin{proof}
\textbf{Step 1, boundedness  of $\mathrm{AT}(\mathrm{I}-\Pi)$:}\\
The operator $\mathrm{AT}(\mathrm{I}-\Pi)$ is bounded if and only if its adjoint $[\mathrm{AT}(\mathrm{I}-\Pi)]^*$ is bounded. Since $\Pi$ is self-adjoint and $\mathrm{T}$ is skew-symmetric, we have
\begin{equation*}
[\mathrm{AT}(\mathrm{I}-\Pi)]^*=-(\mathrm{I}-\Pi)\mathrm{T}^2\Pi[\mathrm{I}+(\mathrm{T} \Pi)^*(\mathrm{T} \Pi)]^{-1}.
\end{equation*} 
Let  $h \in \mathcal{H}$ and $g\colonequals [\mathrm{I}+(\mathrm{T}\Pi)^{*}(\mathrm{T}\Pi)]^{-1}h,$ then
\begin{equation}\label{A*}
[\mathrm{AT}(\mathrm{I}-\Pi)]^*h=-(\mathrm{I}-\Pi)\mathrm{T}^2\Pi g.
\end{equation}
 We compute
\begin{equation}\label{T^2Pg}\mathrm{T}^2\Pi g=\frac{p^T}{p_0}\frac{\partial^2 \Pi g}{\partial x^2}\frac{p}{p_0}-\nabla_x^T V\frac{\partial^2 p_0}{\partial p^2}\nabla_x \Pi g
=\sum_{i,j=1}^d\frac{p_i p_j}{p_0^2}\partial^2_{x_i x_j} \Pi g-\sum_{i,j=1}^d\partial_{x_i} V\partial^2_{p_i p_j}p_0\partial_{x_j} \Pi g.
\end{equation}
 We note that  $\displaystyle \int_{\mathbb{R}^d}\frac{p_i p_j}{p^2_0}Mdp=\int_{\mathbb{R}^d}\partial^2_{p_i p_j}p_0Mdp=0  $ if $i\neq j,$ and we denote $\displaystyle a\colonequals \int_{\mathbb{R}^d}\frac{p^2_i}{p_0^2}Mdp= \frac{1}{d}\int_{\mathbb{R}^d}\frac{|p|^2}{1+|p|^2}Mdp>0.$   Then, using \eqref{T^2Pg} we compute  
\begin{equation}\label{PT^2Pg=div}
\Pi \mathrm{T}^2\Pi g=a\Delta_x \Pi g-a\nabla_x V \cdot \nabla_x \Pi g=\frac{a}{\rho_{\infty}}\text{div}_x( \nabla_x  \Pi g \rho_{\infty} ),
\end{equation}
hence
$$h=g+(\mathrm{T}\Pi)^{*}\mathrm{T}\Pi g=g-\Pi\mathrm{T}^2\Pi g=g-\frac{a}{\rho_{\infty}}\text{div}_x( \nabla_x  \Pi g \rho_{\infty} ).$$
Applying the operator $\Pi$ to this equation we get 
\begin{equation*}\label{ellip.eq. Pg}\Pi g -\frac{a}{\rho_{\infty}}\text{div}_x( \nabla_x  \Pi g \rho_{\infty} )=\Pi h.
\end{equation*}
For the last equation we have the following regularity estimates (see Theorem \ref{lem:ellip} in Appendix \ref{6.1}) \begin{equation}\label{reg}
\int_{\mathbb{R}^d}|\nabla_x \Pi g|^2 |\nabla_x V|^2 \rho_{\infty}dx\leq C_1\int_{\mathbb{R}^d} (\Pi h)^2 \rho_{\infty}dx, \, \, \, \, 
\int_{\mathbb{R}^d}\left| \left|\frac{\partial^2 \Pi g}{\partial x^2} \right|\right|_F^2 \rho_{\infty}dx \leq C_2\int_{\mathbb{R}^d} (\Pi h)^2 \rho_{\infty}dx
\end{equation} for some constants $C_1>0$ and $C_2>0.$

Using \eqref{A*} and \eqref{T^2Pg}  we estimate 
\begin{align}\label{E}
||[\mathrm{AT}(\mathrm{I}-\Pi)]^*h||^2_{L^2(\mathbb{R}^{2d}, f_{\infty})}&=||(\mathrm{I}-\Pi)\mathrm{T}^2\Pi g||^2_{L^2(\mathbb{R}^{2d}, f_{\infty})}\nonumber \leq ||\mathrm{T}^2\Pi g||^2_{L^2(\mathbb{R}^{2d}, f_{\infty})}\nonumber \\
&=\int_{\mathbb{R}^{2d}}\left(\frac{p^T}{p_0}\frac{\partial^2 \Pi g}{\partial x^2}\frac{p}{p_0}-\nabla_x^T V(x)\frac{\partial^2 p_0}{\partial p^2}\nabla_x \Pi g\right)^2f_{\infty}dxdp
\nonumber \\
& \leq 2 \int_{\mathbb{R}^{2d}}\left(\frac{p^T}{p_0}\frac{\partial^2 \Pi g}{\partial x^2}\frac{p}{p_0} \right)^2f_{\infty}dxdp +2\int_{\mathbb{R}^{2d}}\left(\nabla_x^T V(x)\frac{\partial^2 p_0}{\partial p^2}\nabla_x \Pi g\right)^2f_{\infty}dxdp.
\end{align}
Using the H\"older inequality and \eqref{reg} we estimate the last two terms of \eqref{E}:
\begin{equation}\label{E1}
2 \int_{\mathbb{R}^{2d}}\left(\frac{p^T}{p_0}\frac{\partial^2 \Pi g}{\partial x^2}\frac{p}{p_0} \right)^2f_{\infty}dxdp
\leq 2 \int_{\mathbb{R}^{2d}}\frac{|p|^2}{{1+|p|^2}} \left|\left|\frac{\partial^2 \Pi g}{\partial x^2}\right|\right|_F^2f_{\infty}dxdp\leq adC_2 \int_{\mathbb{R}^{d}}(\Pi h)^2\rho_{\infty}dx,
\end{equation}
\begin{multline}\label{E2}
2\int_{\mathbb{R}^{2d}}\left(\nabla_x^T V(x)\frac{\partial^2 p_0}{\partial p^2}\nabla_x \Pi g\right)^2f_{\infty}dxdp
\leq 2\int_{\mathbb{R}^{2d}}\left|\left|\frac{\partial^2 p_0}{\partial p^2}\right|\right|_F|\nabla_x V|^2|\nabla_x \Pi g|^2   f_{\infty}dxdp \leq K_1 \int_{\mathbb{R}^{d}}(\Pi h)^2\rho_{\infty}dx,
\end{multline}
where
 $K_1 \colonequals 2C_1\left(\int_{\mathbb{R}^{d}}\left|\left|\frac{\partial^2 p_0}{\partial p^2}\right|\right|_FMdp\right).$\\
\eqref{E}, \eqref{E1}, and \eqref{E2} show that 
\begin{multline*}
||[\mathrm{AT}(\mathrm{I}-\Pi)]h||^2_{L^2(\mathbb{R}^{2d}, f_{\infty})}=||[\mathrm{AT}(\mathrm{I}-\Pi)]^*h||^2_{L^2(\mathbb{R}^{2d}, f_{\infty})}\\
\leq (adC_2+K_1)||\Pi h||^2_{L^2(\mathbb{R}^{2d}, f_{\infty})}\leq(adC_2+K_1)||h||^2_{L^2(\mathbb{R}^{2d}, f_{\infty})}.
\end{multline*}
This shows that $\mathrm{AT}(\mathrm{I}-\Pi)$ is bounded. Moreover,
replacing $h$  with $(\mathrm{I}-\Pi )h$ and using $(\mathrm{I}-\Pi )^2=(\mathrm{I}-\Pi )$ we obtain 
\begin{equation*}
||[\mathrm{AT}(\mathrm{I}-\Pi)]h||_{L^2(\mathbb{R}^{2d}, f_{\infty})}\leq \sqrt{adC_2+K_1}\,||(\mathrm{I}-\Pi)h||_{L^2(\mathbb{R}^{2d}, f_{\infty})}.
\end{equation*}
\textbf{Step 2, boundedness  of $\mathrm{AL}$:}\\
Let $h \in \mathcal{H}$ and $g\colonequals \mathrm{AL} h.$ Then $$(\mathrm{T}\Pi)^*(\mathrm{L}h)=g+(\mathrm{T}\Pi)^*(\mathrm{T}\Pi)g\, \, \, \, \, \, \Longleftrightarrow \, \, \, \, \, \, g=-\Pi \mathrm{T}(\mathrm{L}h)+\Pi \mathrm{T}^2\Pi g.$$ This shows that $g=\Pi g.$ Using \eqref{PT^2Pg=div} we obtain
\begin{equation}\label{g=PTLh}
g-\frac{a}{\rho_{\infty}}\text{div}_x( \nabla_x   g \rho_{\infty} )=-\Pi \mathrm{T}(\mathrm{L}h).
\end{equation}
Integrating by parts we find 
\begin{align}\label{PTLh}
\Pi \mathrm{T}(\mathrm{L}h)&=\int_{\mathbb{R}^{d}}\left[\frac{p}{p_0}\cdot \nabla_x(\mathrm{L}h)-\nabla_x V\cdot \nabla_p (\mathrm{L}h)\right]Mdp\nonumber \\
&=\int_{\mathbb{R}^{d}}\left[\frac{p}{p_0}\cdot \nabla_x(\mathrm{L}h)-\nabla_x V\cdot \frac{p}{p_0} (\mathrm{L}h)\right]Mdp \nonumber \\
&=\frac{1}{\rho_{\infty}}\text{div}_x\left[\int_{\mathbb{R}^{d}}\frac{p}{p_0} (\mathrm{L}h)f_{\infty}dp\right].
\end{align}
Then, for  $k \in \{1,...,d\}$ and $p=(p_1,...,p_d)^T,$ we compute 
\begin{align}\label{PTLhk}
\int_{\mathbb{R}^{d}}\frac{p_k}{p_0} (\mathrm{L}h)f_{\infty}dp& =\int_{\mathbb{R}^{d}}\frac{p_k}{p_0} \text{div}_p(D\nabla_p h f_{\infty})dp\nonumber \\
& =-\int_{\mathbb{R}^{d}}\nabla^T_p\left(\frac{p_k}{p_0}\right) D\nabla_p h f_{\infty}dp\nonumber \\
&=-\int_{\mathbb{R}^{d}}\left[D\nabla_p\left(\frac{p_k}{p_0}\right)\right]\cdot \nabla_p h f_{\infty}dp\nonumber \\
&=-\int_{\mathbb{R}^{d}}\frac{1}{p_0^2}\partial_{p_k} h f_{\infty}dp 
\nonumber \\
&=-\int_{\mathbb{R}^{d}}\left(\frac{p_k}{p_0^3}+\frac{2p_k}{p_0^4}\right) h f_{\infty}dp .
\end{align}
\eqref{g=PTLh}, \eqref{PTLh}, and \eqref{PTLhk} show that 
\begin{equation}\label{g=PTLh1}
g-\frac{a}{\rho_{\infty}}\text{div}_x( \nabla_x   g \rho_{\infty} )=\frac{1}{\rho_{\infty}}\text{div}_x\left[\int_{\mathbb{R}^{d}}\left(\frac{p}{p_0^3}+\frac{2p}{p_0^4}\right) h f_{\infty}dp\right].
\end{equation}
We multiply this equation by $g\rho_{\infty}$ and integrate by parts
\begin{align*}
\int_{\mathbb{R}^{d}}g^2\rho_{\infty}dx+a\int_{\mathbb{R}^{d}}\nabla^T_x   g \nabla_x   g \rho_{\infty}dx & =-\int_{\mathbb{R}^{d}} \nabla_x  g \cdot\left[\int_{\mathbb{R}^{d}}\left(\frac{p}{p_0^3}+\frac{2p}{p_0^4}\right) h Mdp\right]\rho_{\infty}dx\nonumber \\
& \leq \varepsilon\int_{\mathbb{R}^{d}} |\nabla_x   g|^2 \rho_{\infty}dx+\frac{1}{4 \varepsilon}\int_{\mathbb{R}^{d}} \left|\int_{\mathbb{R}^{d}}\left(\frac{p}{p_0^3}+\frac{2p}{p_0^4}\right) h Mdp\right|^2\rho_{\infty}dx,
\end{align*}
where $\varepsilon>0$ is small enough so that $a-\varepsilon $ is positive. Then by the H\"older inequality 
\begin{equation*}
\int_{\mathbb{R}^{d}}g^2\rho_{\infty}dx+(a-\varepsilon ) \int_{\mathbb{R}^{d}}\nabla^T_x   g\nabla_x   g \rho_{\infty}dx \\
\leq \frac{1}{4 \varepsilon}\left(\int_{\mathbb{R}^{d}}\left|\frac{p}{p_0^3}+\frac{2p}{p_0^4}\right|^2  Mdp\right)\int_{\mathbb{R}^{2d}} h^2 f_{\infty}dxdp.
\end{equation*}
This inequality implies $$||g||_{L^2(\mathbb{R}^{2d}, f_{\infty})}=||\mathrm{AL}h||_{L^2(\mathbb{R}^{2d}, f_{\infty})}\leq K_2||h||_{L^2(\mathbb{R}^{2d}, f_{\infty})}$$ with $K_2 \colonequals \sqrt{\frac{1}{4 \varepsilon}\int_{\mathbb{R}^{d}}\left|\frac{p}{p_0^3}+\frac{2p}{p_0^4}\right|^2  Mdp}. $ This  implies that $\mathrm{AL}h $ is bounded. Moreover, replacing $h$ with $(\mathrm{I}-\Pi)h$ in the estimate above and using $\mathrm{L}\Pi=0,$  we obtain 
$$||\mathrm{AL}h||_{L^2(\mathbb{R}^{2d}, f_{\infty})}\leq K_2||(\mathrm{I}-\Pi)h||_{L^2(\mathbb{R}^{2d}, f_{\infty})}.$$ 
\end{proof}
\begin{Remark}
    We mention that the boundedness of $\mathrm{AL}$ is easily proven for the classical kinetic Fokker-Planck equation \eqref{KFP} with $m=q=\sigma=\nu=1$. The reason is that, in this case, the macroscopic flux $j_h\colonequals \int_{\mathbb{R}^d}vhf_\infty \,dv$ satisfies $j_{\mathrm{L}h}=-j_h$, where $f_\infty(x,v)=const\, e^{-V(x)-|v|^2/2}$ and $\mathrm{L}h:=f_\infty^{-1} \diver_v(\nabla_v h f_\infty)$. This then implies   $\mathrm{AL}=-\mathrm{A},$  see \cite{dolbeault2009hypocoercivity}, \cite[Lemma 20]{VPFP}. However, \eqref{PTLhk}  shows that  this relation does not hold for  the relativistic   flux $\tilde j_h\colonequals \int_{\mathbb{R}^d}\frac{p}{p_0}hf_\infty\,dp.$ Therefore it was more difficult to prove the boundedness of $\mathrm{AL}$ in the relativistic case here.
\end{Remark}
\begin{proof}[\textbf{Proof of Theorem \ref{main result}}]
Let $f$ be the solution of \eqref{Eq}. Then $\displaystyle h\colonequals \frac{f-f_{\infty}}{f_{\infty}}$ satisfies 
\begin{equation}\label{t+T+L}\partial_t h+\mathrm{T}h=\mathrm{L} h, \, \, \, \, \, \, \, \, \, h_{|t=0}=h_0,
\end{equation}
where the operators $\mathrm{T}$ and $\mathrm{L}$ are  defined in \eqref{T} and \eqref{L}, respectively. If $\Pi$ is defined as in \eqref{P}, Proposition \ref{proporties} and Lemma \ref{LEMMA} show that these operators  satisfy the conditions \eqref{mic.coer}-\eqref{boun.aux}. Therefore, Theorem \ref{DMS} holds for \eqref{t+T+L}, and \eqref{|| decay} provides the claimed result.
\end{proof}

\section{Exponential convergence in $\mathscr{H}^1(\mathbb{R}^{2d}, f_{\infty})$}\label{sec:H1-conv}
In this section, we shall study the long-time behavior of \eqref{Eqh} in $\mathscr{H}^1(\mathbb{R}^{2d}, f_{\infty}).$  To do that we shall construct another Lyapunov functional (rather than $\mathrm{H}_{\delta}$ which is used  to prove Theorem \ref{main result}).
\subsection{Preliminaries}
\begin{lemma}\label{lemma ||h||^2}
Let $h$ be the solution of \eqref{Eqh}. Then, for all $t>0,$ 
\begin{equation*}
\frac{d}{dt} \int_{ \mathbb{R}^{2d}}h^2 f_{\infty}dxdp=-2\int_{ \mathbb{R}^{2d}}\nabla^{T}_p h D\nabla_p h f_{\infty}dxdp.
\end{equation*}
In particular, we have $||h(t)||_{L^2(\mathbb{R}^{2d}, f_{\infty})}\leq ||h_0||_{L^2(\mathbb{R}^{2d}, f_{\infty})}$ for all $t\geq 0.$
\end{lemma}
\begin{proof}
We integrate by parts and use $\nabla_pf_{\infty}=-\frac{p}{p_0} f_{\infty}$ and $\nabla_xf_{\infty}=-\nabla_x V f_{\infty}$ to obtain 
\begin{align*}
\frac{d}{dt} \int_{\mathbb{R}^{2d}}h^2f_{\infty}dxdp&=2\int_{ \mathbb{R}^{2d}}h\partial_t h f_{\infty}dxdp\\
&=-2\int_{ \mathbb{R}^{2d}}\left(\frac{p}{p_0}\cdot \nabla_x h-\nabla_x V\cdot \nabla_p h\right)hf_{\infty}dxdp +2\int_{ \mathbb{R}^{2d}}\text{div}_p(D\nabla_p h f_{\infty})hdxdp\\
&=-2\int_{ \mathbb{R}^{2d}}\nabla^{T}_p h D\nabla_p h f_{\infty}dxdp.
\end{align*} 
\end{proof}
Let  $ P=P(x, p) \in \mathbb{R}^{2d\times 2d}$  be a  symmetric, positive definite matrix depending on the variables $x, \,p\in \mathbb{R}^d$ and specified later.  We define 
\begin{equation}\label{S_P}
\mathrm{S}_P[h]\colonequals \int_{\mathbb{R}^{2d}}\begin{pmatrix}
\nabla_x h\\ \nabla_p h
\end{pmatrix}^T  P\begin{pmatrix}
\nabla_x h\\
\nabla_p h
\end{pmatrix} f_{\infty}dxdp.
\end{equation}

\begin{lemma}\label{lemma main}
Let $h$ be the  solution of \eqref{Eqh}. Then, for all $t>0,$ 
\begin{multline}\label{derivS} 
\frac{d}{dt} \mathrm{S}_P[h(t)]=-2\int_{ \mathbb{R}^{2d}}\left\{\sum_{i, j=1}^d \begin{pmatrix}
\nabla_x ( \partial_{p_i} h)\\ \nabla_p(\partial_{p_i} h)
\end{pmatrix}^T  P\begin{pmatrix}
\nabla_x (\partial_{p_j} h)\\\nabla_p(\partial_{p_j} h)
\end{pmatrix} a_{ij}\right\}f_{\infty}dxdp\\+2\int_{ \mathbb{R}^{2d}}  \begin{pmatrix}
\nabla_x h\\ \nabla_p h
\end{pmatrix}^T P \begin{pmatrix} 0\\\sum_{i,j=1}^d \nabla_p a_{ij}\partial^2_{p_ip_j}h
\end{pmatrix} f_{\infty}dxdp\\
-\int_{ \mathbb{R}^{2d}}  \begin{pmatrix}
\nabla_x h\\ \nabla_p h
\end{pmatrix}^T  \left\{QP+PQ^T\right\}\begin{pmatrix}
\nabla_x h\\ \nabla_p h
\end{pmatrix} f_{\infty}dxdp\\
+\int_{\mathbb{R}^{2d}}  \begin{pmatrix}
\nabla_x h\\ \nabla_p h
\end{pmatrix}^T  \left\{\sum_{i=1}^d\left(\frac{p_i}{p_0} \partial_{x_i} P -\partial_{x_i} V  \partial_{p_i} P\right)+\sum_{i,j=1}^d \frac{1}{f_{\infty}}\partial_{p_j}(\partial_{p_i}Pa_{ij}f_{\infty})\right\}\begin{pmatrix}
\nabla_x h\\ \nabla_p h
\end{pmatrix} f_{\infty}dxdp,
\end{multline}
where $Q=Q(x,p)\colonequals \begin{pmatrix} 
0&\frac{1}{p_0}(I-\frac{p\otimes p}{p_0^2})\\
-\frac{\partial^2 V}{\partial x^2 }&I-\frac{d}{p_0}(I-\frac{p \otimes p}{p_0^2})
\end{pmatrix}$ and $\displaystyle a_{ij}\colonequals \frac{\delta_{ij}+p_i p_j}{p_0}$  (which are the elements of $ \displaystyle D=D( p)=\frac{I+p\otimes p}{p_0}$ and $\delta_{ij}$ is the Kronecker symbol).
\end{lemma}
\begin{proof} We write \eqref{Eqh} as
\begin{align}\label{th}
\partial_t h&=\frac{1}{f_{\infty}}\text{div}_{p}(D\nabla_p h f_{\infty})-\frac{p}{p_0}\cdot\nabla_x h+\nabla_x V \cdot \nabla_p h\nonumber \\&=\sum_{i,j=1}^da_{ij}\partial^2_{p_ip_j}h-\sum_{i,j=1}^da_{ij}\frac{p_i}{p_0}\partial_{p_j}h+\sum_{i,j=1}^d\partial_{p_i}a_{ij}\partial_{p_j}h \nonumber-\frac{p}{p_0}\cdot\nabla_x h+\nabla_x V \cdot \nabla_p h\\&=\sum_{i,j=1}^da_{ij}\partial^2_{p_ip_j}h-p\cdot \nabla_p h+\frac{d p}{p_0}\cdot \nabla_p h-\frac{p}{p_0}\cdot\nabla_x h+\nabla_x V \cdot \nabla_p h,
\end{align}
where we used 
\begin{equation}\label{aux eq}
\sum_{i=1}^d a_{ij} \frac{p_i}{p_0}=p_j,\, \, \, \, \, \sum_{i=1}^d \partial_{p_i}a_{ij}= \frac{d p_j}{p_0}. 
\end{equation}
We denote $u\colonequals \begin{pmatrix}
\nabla_x h\\ \nabla_p h
\end{pmatrix}, $ $\displaystyle u_1\colonequals \nabla_x h,$ $\displaystyle u_2\colonequals \nabla_p h.$ We get from \eqref{th} that  $u_1$ and $u_2$ satisfy 
\begin{equation*}\partial_t u_1= \sum_{i,j=1}^d a_{ij}\partial_{p_ip_j}^2 u_1-\sum_{j=1}^d p_j \partial_{p_j}u_1+\sum_{j=1}^d \frac{dp_j}{p_0}\partial_{p_j}u_1-\sum_{j=1}^d \frac{p_j}{p_0}\partial_{x_j}u_1 +\sum_{j=1}^d\partial_{x_j} V\partial_{p_j}u_1+ \frac{\partial^2 V}{\partial x^2}u_2,
\end{equation*}
\begin{multline*}
 \partial_t u_2=\sum_{i,j=1}^d a_{ij}\partial_{p_ip_j}^2 u_2-\sum_{j=1}^d p_j \partial_{p_j}u_2+\sum_{j=1}^d \frac{d p_j}{p_0}\partial_{p_j}u_2-\sum_{j=1}^d \frac{p_j}{p_0}\partial_{x_j}u_2+\sum_{j=1}^d \partial_{x_j}V\partial_{p_j}u_2\\+ \sum_{i,j=1}^d \nabla_p a_{ij}\partial_{p_ip_j}^2 h-\left(I-\frac{d}{p_0}\left(I-\frac{p\otimes p}{p_0^2}\right)\right)u_2-\frac{1}{p_0}\left(I-\frac{p\otimes p}{p_0^2}\right) u_1.
 \end{multline*}
 These equations can be written  with respect to $u=\begin{pmatrix}
 u_1\\u_2
\end{pmatrix}:  $
 \begin{multline*}\partial_t u=\sum_{i,j=1}^d a_{ij}\partial_{p_ip_j}^2 u-\sum_{j=1}^d p_j \partial_{p_j}u+\sum_{j=1}^d \frac{d p_j}{p_0}\partial_{p_j}u-\sum_{j=1}^d \frac{p_j}{p_0}\partial_{x_j}u+ \sum_{j=1}^d\partial_{x_j} V\partial_{p_j}u\\-Q^Tu+\begin{pmatrix} 0\\\sum_{i,j=1}^d \nabla_p a_{ij}\partial^2_{p_ip_j}h
\end{pmatrix}.
\end{multline*}
 It allows us to compute the time derivative 
   \begin{align}\label{time.derS}
 \frac{d}{dt} \mathrm{S}_P[h(t)]&=2\int_{\mathbb{R}^{2d}}
  u^T P  \partial_t u
 f_{\infty}dxdp\nonumber \\
 &= 2\sum_{i,j=1}^d\int_{ \mathbb{R}^{2d}} u^T P
 \partial_{p_ip_j}^2 u a_{ij} f_{\infty}dxdp \nonumber \\&  \, \, \, \, \, \,-2  \sum_{j=1}^d\int_{\mathbb{R}^{2d}} u^T P \partial_{p_j}u p_j f_{\infty}dxdp+ 2d\sum_{j=1}^d \int_{ \mathbb{R}^{2d}} u^TP
\partial_{p_j}u \frac{p_j}{p_0} f_{\infty}dxdp\nonumber \\& \, \, \, \, \, \, -2\sum_{j=1}^d\int_{\mathbb{R}^{2d}} u^TP
   \partial_{x_j}u \frac{p_j}{p_0} f_{\infty}dxdp+2\sum_{j=1}^d\int_{\mathbb{R}^{2d}} u^TP
   \partial_{p_j}u \partial_{x_j}V f_{\infty}dxdp\nonumber \\ &  \, \, \, \, \, \,-\int_{ \mathbb{R}^{2d}} u^T \{QP+PQ^T\}u f_{\infty}dxdp+2\int_{ \mathbb{R}^{2d}} u^T P\begin{pmatrix}0\\\sum_{i,j=1}^d \nabla_p a_{ij}\partial^2_{p_ip_j}h
\end{pmatrix} f_{\infty}dxdp.
 \end{align}
  First, we consider the term in the second line of \eqref{time.derS}, and we integrate by parts using  \eqref{aux eq} and $\partial_{p_i}f_{\infty}=-\frac{p_i}{p_0}f_{\infty}:$
  \begin{align}\label{term1}
  2\sum_{i,j=1}^d\int_{ \mathbb{R}^{2d}}& u^T P
 \partial_{p_ip_j}^2 u a_{ij} f_{\infty}dxdp\nonumber \\
& =-2\sum_{i,j=1}^d\int_{ \mathbb{R}^{2d}} (\partial_{p_i}u)^T P
 \partial_{p_j} u a_{ij} f_{\infty}dxdv-2\sum_{i,j=1}^d\int_{ \mathbb{R}^{2d}} u^T P
 \partial_{p_j} u \partial_{p_i}a_{ij} f_{\infty}dxdp
 \nonumber \\&\, \, \, \, \, \,\, +2\sum_{i,j=1}^d\int_{ \mathbb{R}^{2d}} u^T P
 \partial_{p_j} u a_{ij} \frac{p_i}{p_0} f_{\infty}dxdp
 -2\sum_{i,j=1}^d\int_{ \mathbb{R}^{2d}} u^T \partial_{p_i}P
 \partial_{p_j} u a_{ij} f_{\infty}dxdp\nonumber \\
&=-2\sum_{i,j=1}^d\int_{\mathbb{R}^{2d}} (\partial_{p_i}u)^T P
 \partial_{p_j} u a_{ij} f_{\infty}dxdv-2\sum_{j=1}^d\int_{ \mathbb{R}^{2d}} u^T P
 \partial_{p_j} u \frac{dp_j}{p_0} f_{\infty}dxdp\nonumber \\
 &\, \, \, \, \, \,\, +2\sum_{j=1}^d\int_{ \mathbb{R}^{2d}} u^T P
 \partial_{p_j} u p_j f_{\infty}dxdp-2\sum_{i,j=1}^d\int_{\mathbb{R}^{2d}} u^T \partial_{p_i}P
 \partial_{p_j} u a_{ij} f_{\infty}dxdp.
 \end{align}
 We now compute the last integral in \eqref{term1} by integrating by parts
 \begin{multline*}
- 2\sum_{i,j=1}^d\int_{ \mathbb{R}^{2d}} u^T \partial_{p_i}P
 \partial_{p_j} u a_{ij} f_{\infty}dxdp\\
 =2\sum_{i,j=1}^d\int_{ \mathbb{R}^{2d}} ( \partial_{p_j}u)^T \partial_{p_i}P
 u a_{ij} f_{\infty}dxdp+2\sum_{i,j=1}^d\int_{ \mathbb{R}^{2d}} u^T \partial_{p_j}(\partial_{p_i}P a_{ij} f_{\infty}) udxdp.
 \end{multline*}
Since $P$ is symmetric and $a_{ij}=a_{ji},$ this equation implies \begin{equation}\label{derterm}
- 2\sum_{i,j=1}^d\int_{ \mathbb{R}^{2d}} u^T \partial_{p_i}P
 \partial_{p_j} u a_{ij} f_{\infty}dxdp=\int_{\mathbb{R}^{2d}} u^T \left(\sum_{i,j=1}^d\frac{1}{f_{\infty}}\partial_{p_j}(\partial_{p_i}P a_{ij} f_{\infty})\right) uf_{\infty}dxdp.
\end{equation}  
\eqref{time.derS}, \eqref{term1}, and \eqref{derterm} show that  the sum of the terms in the second and third lines of \eqref{time.derS} equals
\begin{multline}\label{1im}
2\sum_{i,j=1}^d\int_{ \mathbb{R}^{2d}} u^T P
 \partial_{p_ip_j}^2 u a_{ij} f_{\infty}dxdp\\-2  \sum_{j=1}^d\int_{\mathbb{R}^{2d}} u^T P \partial_{p_j}u p_j f_{\infty}dxdp+ 2d\sum_{j=1}^d \int_{ \mathbb{R}^{2d}} u^TP
\partial_{p_j}u p_j\frac{1}{p_0} f_{\infty}dxdp
\\=-2\sum_{i,j=1}^d\int_{ \mathbb{R}^{2d}} (\partial_{p_i}u)^T P
 \partial_{p_j} u a_{ij} f_{\infty}dxdv+\int_{ \mathbb{R}^{2d}} u^T \left(\sum_{i,j=1}^d\frac{1}{f_{\infty}}\partial_{p_j}(\partial_{p_i}P a_{ij} f_{\infty})\right) uf_{\infty}dxdp.
 \end{multline}
  We consider  the terms in the fourth line of \eqref{time.derS} and itegrate by parts in both terms:
   \begin{multline*}
 -2\sum_{j=1}^d\int_{\mathbb{R}^{2d}} u^TP
   \partial_{x_j}u \frac{p_j}{p_0} f_{\infty}dxdp+2\sum_{j=1}^d\int_{\mathbb{R}^{2d}} u^TP
   \partial_{p_j}u \partial_{x_j}V f_{\infty}dxdp\\
   =2\sum_{j=1}^d\int_{\mathbb{R}^{2d}} (\partial_{x_j}u)^TP
u \frac{p_j}{p_0} f_{\infty}dxdp+ 2\sum_{j=1}^d\int_{\mathbb{R}^{2d}} u^T\partial_{x_j}P
u \frac{p_j}{p_0} f_{\infty}dxdp\\
   -2\sum_{j=1}^d\int_{\mathbb{R}^{2d}} (\partial_{p_j}u)^TP
u \partial_{x_j}V f_{\infty}dxdp-2\sum_{j=1}^d\int_{\mathbb{R}^{2d}} u^T\partial_{p_j}P
u \partial_{x_j}V f_{\infty}dxdp.
 \end{multline*}
Since $P$ is symmetric, we conclude from this last equation: 
\begin{multline}\label{2im}
 -2\sum_{j=1}^d\int_{\mathbb{R}^{2d}} u^TP
   \partial_{x_j}u \frac{p_j}{p_0} f_{\infty}dxdp+2\sum_{j=1}^d\int_{\mathbb{R}^{2d}} u^TP
   \partial_{p_j}u \partial_{x_j}V f_{\infty}dxdp\\
   =\sum_{j=1}^d\int_{\mathbb{R}^{2d}} u^T\left(\partial_{x_j}P
 \frac{p_j}{p_0}-\partial_{p_j}P
 \partial_{x_j}V\right)u f_{\infty}dxdp.
\end{multline}
\eqref{time.derS}, \eqref{1im}, and \eqref{2im} yield the claimed equality \eqref{derivS}.
\end{proof}
\subsection{The second Lyapunov functional}
Let  $\mathrm{H}_{\delta}$  and $\mathrm{S}_{P}$ be, respectively, the functionals defined in \eqref{DMS funct} and \eqref{S_P}. For some $\gamma>0,$ we define the functional 
\begin{equation*}
\mathrm{E}[h]\colonequals \gamma ||h||^2_{L^2(\mathbb{R}^{2d}, f_{\infty})}+\mathrm{H}_{\delta}[h]+\mathrm{S}_P[h].
\end{equation*}
 It is clear that $\mathrm{E}$ depends on the parameters $\gamma,\, \delta,$  and the matrix $P.$ Let $\delta_0$ be given in \eqref{delta0} and choose any $\delta\in (0, \delta_0).$ Then the decay estimates of Theorem \ref{DMS}  holds for the relativistic Fokker-Planck equation \eqref{Eqh}.   Our goal is to choose  a suitable $\gamma>0$ and  $P$ so that $\mathrm{E}[h]$ is equivalent to $||h||^2_{\mathscr{H}^1(\mathbb{R}^{2d}, f_{\infty})}$  and satisfies a Gr\"onwall inequality (see \eqref{dt E<-labmda E} below) for the solution $h$ of \eqref{Eqh}.  We choose 
\begin{equation}\label{Pmat}
P=P(x,p)\colonequals \begin{pmatrix}
\frac{2\varepsilon^3}{V_0^3p_0^3}(I- \frac{p\otimes p}{p_0^2})& \frac{\varepsilon^2}{V_0^2p_0^2}I\\
\frac{\varepsilon^2}{V_0^2p_0^2} I& \frac{2\varepsilon}{V_0 p_0} (I+p\otimes p)
\end{pmatrix},
\end{equation}
where  $\varepsilon$ is a positive constant which will be fixed later. 
 We note the matrices $I- \frac{p\otimes p}{p_0^2}$ and $I+p\otimes p$ are positive definite, and $
 \left(I- \frac{p\otimes p}{p_0^2}\right)^{-1}=I+p\otimes p.$
  This helps to check that $P$ is positive definite for all $x,\, p\in \mathbb{R}^d$ and 
\begin{equation}\label{.<P<.}0<\begin{pmatrix}
\frac{ \varepsilon^3}{V_0^3p_0^{3}} (I-\frac{p\otimes p}{p_0^2})&0\\0&   \frac{\varepsilon}{V_0 p_0} (I+p\otimes p)
\end{pmatrix}\leq P\leq \begin{pmatrix}
\frac{3 \varepsilon^3}{V_0^3p_0^{3}} (I-\frac{p\otimes p}{p_0^2})&0\\0&  \frac{3\varepsilon}{V_0p_0} (I+p\otimes p)
\end{pmatrix}. 
\end{equation}
 We now state the main result of this section.
\begin{theorem}\label{th:hypocoercivity}  
Let Assumption \ref{Assumptions i-ii} hold and  $h$ be the solution of \eqref{Eqh} with  initial data  $h_0 \in \mathscr{H}^1( \mathbb{R}^{2d}, f_{\infty})
$ such that $\int_{ \mathbb{R}^{2d}}h_0 f_{\infty} dxdp=0.$
If $\varepsilon>0$ in \eqref{Pmat} is small enough, then
\begin{equation}\label{dt E<-labmda E}
\frac{d}{dt}\mathrm{E}[h(t)]\leq -2\Lambda \mathrm{E}[h(t)], \, \, \, \, \, \forall\, t> 0
\end{equation} holds for a positive constant $\Lambda$ (independent of $h_0$).
 In particular, \begin{equation}\label{Cor:H^1 decay}\mathrm{E}[h(t)]\leq e^{-2\Lambda t} \mathrm{E}[h_0], \, \, \, \, \,\forall\,  t\geq 0.
 \end{equation}  
\end{theorem}
\begin{proof}
Theorem\,\ref{DMS} provides 
\begin{equation*}\frac{d}{dt}\mathrm{H}_{\delta}[h(t)]\leq -2\lambda \mathrm{H}_{\delta}[h(t)] \leq -\frac{\lambda (2-\delta)}{2}\int_{ \mathbb{R}^{2d}} h^2 f_{\infty}dxdp. 
\end{equation*}
 This estimate, Lemma \ref{lemma ||h||^2}, and Lemma \ref{lemma main} show that 
\begin{multline}\label{dt E}
\frac{d }{dt}\mathrm{E}[h(t)]\leq -\frac{\lambda (2-\delta)}{2}\int_{ \mathbb{R}^{2d}} h^2 f_{\infty}dxdp \\
-2\int_{ \mathbb{R}^{2d}}\left\{\sum_{i, j=1}^d \begin{pmatrix}
\nabla_x ( \partial_{p_i} h)\\ \nabla_p(\partial_{p_i} h)
\end{pmatrix}^T  P\begin{pmatrix}
\nabla_x (\partial_{p_j} h)\\\nabla_p(\partial_{p_j} h)
\end{pmatrix} a_{ij}\right\}f_{\infty}dxdp\\+2\int_{ \mathbb{R}^{2d}}  \begin{pmatrix}
\nabla_x h\\ \nabla_p h
\end{pmatrix}^T P \begin{pmatrix} 0\\\sum_{i,j=1}^d \nabla_p a_{ij}\partial^2_{p_ip_j}h
\end{pmatrix} f_{\infty}dxdp\\
-\int_{ \mathbb{R}^{2d}}  \begin{pmatrix}
\nabla_x h\\ \nabla_p h
\end{pmatrix}^T  \left\{2\gamma\begin{pmatrix}
0& 0\\
0& D
\end{pmatrix}+QP+PQ^T\right\}\begin{pmatrix}
\nabla_x h\\ \nabla_p h
\end{pmatrix} f_{\infty}dxdp\\
+\int_{\mathbb{R}^{2d}}  \begin{pmatrix}
\nabla_x h\\ \nabla_p h
\end{pmatrix}^T  \left\{\sum_{i=1}^d\left(\frac{p_i}{p_0} \partial_{x_i} P -\partial_{x_i} V  \partial_{p_i} P\right)+\sum_{i,j=1}^d \frac{1}{f_{\infty}}\partial_{p_j}(\partial_{p_i}Pa_{ij}f_{\infty})\right\}\begin{pmatrix}
\nabla_x h\\ \nabla_p h
\end{pmatrix} f_{\infty}dxdp.
\end{multline}

 \subsubsection*{Step 1, estimates on the second order derivatives:}
 We first consider the term in the third line of \eqref{dt E}: 
 \begin{multline}\label{21}
 2\int_{ \mathbb{R}^{2d}} \begin{pmatrix}
\nabla_x h\\ \nabla_p h
\end{pmatrix}^T P \begin{pmatrix} 0\\ \sum_{i,j=1}^d \nabla_p a_{ij}\partial^2_{p_ip_j}h
\end{pmatrix} f_{\infty}dxdp\\
=2\varepsilon^2 \sum_{i,j=1}^d\int_{ \mathbb{R}^{2d}} 
\frac{1}{V^2_0p_0^2}\nabla_x h \cdot \nabla_p a_{ij}\partial^2_{p_ip_j}h
 f_{\infty}dxdp+ 4\varepsilon \sum_{i,j=1}^d\int_{\mathbb{R}^{2d}} 
\frac{1}{V_0p_0}\nabla_p^T h (I+p\otimes p) \nabla_p a_{ij}\partial^2_{p_ip_j}h
 f_{\infty}dxdp.
 \end{multline}
 Let $w\colonequals \begin{pmatrix}
  \nabla_p (\partial_{p_1}h)\\ \vdots\\  \nabla_p (\partial_{p_d}h)
 \end{pmatrix}\in \mathbb{R}^{d^2}$ and $z\colonequals \begin{pmatrix}
  \nabla_x h \cdot \nabla_p a_{11}\\ \nabla_x h \cdot \nabla_p a_{12}\\ \vdots\\ \nabla_x h \cdot \nabla_p a_{dd} 
 \end{pmatrix}\in \mathbb{R}^{d^2}. $
 Using the the relation  $(D\otimes D)^{-1}=D^{-1}\otimes D^{-1}>0$ (see \cite[Corollary 4.2.11]{MA} ) and applying \eqref{Mat.ineq}  (from Appendix) to $w$ and $z,$ we obtain 
 \begin{multline}\label{22}
 2\varepsilon^2 \sum_{i,j=1}^d\int_{ \mathbb{R}^{2d}} 
\frac{1}{V^2_0p_0^2}\nabla_x h \cdot \nabla_p a_{ij}\partial^2_{p_ip_j}h
 f_{\infty}dxdp=2\varepsilon^2 \int_{ \mathbb{R}^{2d}} 
\frac{1}{V^2_0p_0^2} z\cdot w
 f_{\infty}dxdp\\
 \leq \frac{\varepsilon^3}{\eta} \int_{ \mathbb{R}^{2d}} 
\frac{1}{V_0^3p_0^3}z^T D^{-1}\otimes D^{-1}z  f_{\infty}dxdp+\varepsilon \eta  \int_{\mathbb{R}^{2d}} 
\frac{1}{V_0p_0} w^T D\otimes Dw
  f_{\infty}dxdp,
 \end{multline}
 where $\eta>0$  will be fixed later.
 We use $D^{-1}={p_0}\left(I-\frac{p\otimes p}{p_0^2}\right)$ and \eqref{sum<d()}  (from Appendix)  to estimate  
 \begin{align*}
 z^T D^{-1}\otimes D^{-1}z &=\sum _{k,l,i,j=1}^d{p_0^2} \left(\delta_{kl}-\frac{p_kp_l}{p_0^2}\right) \left(\delta_{ij}-\frac{p_ip_j}{p_0^2}\right)(\nabla_x h \cdot \nabla_p a_{lj})(\nabla_x h \cdot \nabla_p a_{ki})\\
 &=\nabla_x^T h\left\{\sum _{k,l,i,j=1}^d {p_0^2} \left(\delta_{kl}-\frac{p_kp_l}{p_0^2}\right) \left(\delta_{ij}-\frac{p_ip_j}{p_0^2}\right)\nabla_p a_{lj}\otimes  \nabla_p a_{ki}\right\}\nabla_x h\\
 &\leq d \nabla_x^T h \left(I-\frac{p\otimes p}{p_0^2}\right) \nabla_x h.
 \end{align*}
 \eqref{22} and the last estimate imply
 \begin{multline}\label{222}
 2\varepsilon^2 \sum_{i,j=1}^d\int_{ \mathbb{R}^{2d}} 
\frac{1}{V^2_0p_0^2}\nabla_x h \cdot \nabla_p a_{ij}\partial^2_{p_ip_j}h
 f_{\infty}dxdp\\
 \leq \frac{\varepsilon^3d}{\eta} \int_{ \mathbb{R}^{2d}} 
\frac{1}{V_0^3p_0^3}\nabla_x^T h \left(I-\frac{p\otimes p}{p_0^2}\right) \nabla_x h f_{\infty}dxdp+\varepsilon \eta  \int_{\mathbb{R}^{2d}} 
\frac{1}{V_0p_0} w^T D\otimes Dw
  f_{\infty}dxdp.
 \end{multline}
 We now work on the last term of \eqref{21}. We define $$z_1\colonequals \begin{pmatrix}
 \nabla_p^T h (I+p\otimes p) \nabla_p a_{11}\\ \nabla_p^T h (I+p\otimes p) \nabla_p a_{12}\\ \vdots\\ \nabla_p^T h (I+p\otimes p) \nabla_p a_{dd} 
 \end{pmatrix}\in \mathbb{R}^{d^2}. $$
 Similar to \eqref{22}, we estimate 
  \begin{multline}\label{23}
 4\varepsilon \sum_{i,j=1}^d\int_{\mathbb{R}^{2d}} 
\frac{1}{V_0p_0}\nabla_p^T h (I+p\otimes p) \nabla_p a_{ij}\partial^2_{p_ip_j}h
 f_{\infty}dxdp= 4\varepsilon\int_{\mathbb{R}^{2d}} 
\frac{1}{V_0p_0} z_1\cdot w
 f_{\infty}dxdp\\
 \leq \frac{2\varepsilon }{\eta}\int_{\mathbb{R}^{2d}} 
\frac{1}{V_0p_0} z_1^TD^{-1}\otimes D^{-1}z_1
 f_{\infty}dxdp+ 2 \varepsilon \eta \int_{\mathbb{R}^{2d}}\frac{1}{V_0p_0} 
w^TD\otimes D w
  f_{\infty}dxdp.
 \end{multline} 
Using  \eqref{sum d(I+pop)}  (from Appendix) we estimate  
 \begin{align*}
 z_1^T D^{-1}&\otimes D^{-1}z_1 =\sum _{k,l,i,j=1}^d{p_0^2} \left(\delta_{kl}-\frac{p_kp_l}{p_0^2}\right) \left(\delta_{ij}-\frac{p_ip_j}{p_0^2}\right)( \nabla_p^T h (I+p\otimes p) \nabla_p a_{lj})( \nabla_p^T h (I+p\otimes p) \nabla_p a_{ki})\\
 &=\nabla_p^T h\left\{\sum _{k,l,i,j=1}^d {p_0^2} \left(\delta_{kl}-\frac{p_kp_l}{p_0^2}\right) \left(\delta_{ij}-\frac{p_ip_j}{p_0^2}\right)  ((I+p\otimes p) \nabla_p a_{lj})\otimes   ( (I+p\otimes p) \nabla_p a_{ki})\right\}\nabla_p h\\
 &\leq d \nabla_p^T h \left(I+p\otimes p\right) \nabla_p h.
 \end{align*}
 \eqref{23} and the last estimate imply
 \begin{multline}\label{233}
 4\varepsilon \sum_{i,j=1}^d\int_{\mathbb{R}^{2d}} 
\frac{1}{V_0p_0}\nabla_p^T h (I+p\otimes p) \nabla_p a_{ij}\partial^2_{p_ip_j}h
 f_{\infty}dxdp\\
 \leq \frac{2\varepsilon d}{\eta}\int_{\mathbb{R}^{2d}} 
\frac{1}{V_0p_0} \nabla_p^T h \left(I+p\otimes p\right) \nabla_p h
 f_{\infty}dxdp+ 2 \varepsilon \eta  \int_{\mathbb{R}^{2d}}\frac{1}{V_0p_0} 
w^TD\otimes D w
  f_{\infty}dxdp.
 \end{multline} 
Then \eqref{21}, \eqref{222}, and \eqref{233} imply 
\begin{multline}\label{24}
 2\int_{ \mathbb{R}^{2d}} \begin{pmatrix}
\nabla_x h\\ \nabla_p h
\end{pmatrix}^T P \begin{pmatrix} 0\\ \sum_{i,j=1}^d \nabla_p a_{ij}\partial^2_{p_ip_j}h
\end{pmatrix} f_{\infty}dxdp\\ 
\leq  \int_{ \mathbb{R}^{2d}}\begin{pmatrix}
\nabla_x  h\\ \nabla_p h
\end{pmatrix}^T  
\begin{pmatrix}
\frac{\varepsilon^3 d}{\eta V_0^3 p_0^3}(I-\frac{p\otimes p}{p_0^2})& 0\\
0& \frac{2\varepsilon d}{\eta V_0p_0} (I+p\otimes p)
\end{pmatrix} \begin{pmatrix}
\nabla_x  h\\ \nabla_p h
\end{pmatrix} f_{\infty}dxdp\\+ 3 \varepsilon \eta \int_{\mathbb{R}^{2 d}}\frac{1}{V_0p_0} 
w^TD\otimes D w
  f_{\infty}dxdp.
\end{multline}
 Let $u\colonequals \begin{pmatrix}
 \nabla_x (\partial_{p_1}h)\\ \nabla_p (\partial_{p_1}h)\\ \vdots\\ \nabla_x (\partial_{p_d}h)\\ \nabla_p (\partial_{p_d}h)
 \end{pmatrix}\in \mathbb{R}^{2d^2}.$
and $\tilde{P}\colonequals D\otimes P=\begin{pmatrix}
a_{11}P& \cdots & a_{1d}P
\\
\cdots&\cdots&\cdots&\\
a_{1d}P& \cdots & a_{dd}P
\end{pmatrix}\in \mathbb{R}^{2d^2\times 2d^2}.
$
Then we can write 
\begin{equation}\label{refor}
-2\int_{ \mathbb{R}^{2d}}\left\{\sum_{i, j=1}^d \begin{pmatrix}
\nabla_x ( \partial_{p_i} h)\\ \nabla_p(\partial_{p_i} h)
\end{pmatrix}^T  P\begin{pmatrix}
\nabla_x (\partial_{p_j} h)\\\nabla_p(\partial_{p_j} h)
\end{pmatrix} a_{ij}\right\}f_{\infty}dxdp=-2\int_{ \mathbb{R}^{2d}}u^T\tilde{P}uf_{\infty}dxdp.
\end{equation}
Since $P$ and $D$ are positive definite, $D \otimes P$ is also positive definite, see  \cite[Corollary 4.2.13]{MA}. Moreover, $P$ can be written as a sum of two positive semi-definite matrices: 
$$P= \begin{pmatrix}
\frac{2\varepsilon^3}{V_0^3p_0^{3}} (I-\frac{p\otimes p}{p_0^2})& \frac{\varepsilon^2}{V^2_0p_0^2}I\\
\frac{\varepsilon^2}{V_0^2p_0^2} I&  \frac{\varepsilon}{V_0p_0}(I+p\otimes p)
\end{pmatrix}+\begin{pmatrix}
0&0\\
0& \frac{\varepsilon}{V_0p_0}(I+p\otimes p)
\end{pmatrix}. $$ This implies
$$\tilde{P}=D\otimes P\geq  D \otimes \begin{pmatrix}
0&0\\
0& \frac{\varepsilon}{V_0p_0}(I+p\otimes p)
\end{pmatrix}. $$
This inequality and \eqref{refor} show that 
\begin{multline}\label{25}
-2\int_{\mathbb{R}^{2d}}\left\{\sum_{i, j=1}^d \begin{pmatrix}
\nabla_x ( \partial_{p_i} h)\\ \nabla_p(\partial_{p_i} h)
\end{pmatrix}^T   P\begin{pmatrix}
\nabla_x (\partial_{p_j} h)\\\nabla_p(\partial_{p_j} h)
\end{pmatrix} a_{ij}\right\}f_{\infty}dxdp\\ \leq -2\int_{\mathbb{R}^{2d}}u^T D \otimes \begin{pmatrix}
0&0\\
0& \frac{\varepsilon}{V_0p_0}(I+p\otimes p)
\end{pmatrix} uf_{\infty}dxdp=-2\varepsilon \int_{ \mathbb{R}^{2d}} \frac{1}{V_0}w^T D \otimes D
 wf_{\infty}dxdp.
\end{multline}
We choose  any $\eta\in (0, \frac{2}{3}],$ and hence $2-\frac{3\eta}{p_0}\geq 0$  for all $p\in \mathbb{R}^d.$ Then \eqref{25} and \eqref{24} yield
 \begin{multline}\label{26}
 -2\int_{ \mathbb{R}^{2d}}\left\{\sum_{i, j=1}^d \begin{pmatrix}
\nabla_x ( \partial_{p_i} h)\\ \nabla_p(\partial_{p_i} h)
\end{pmatrix}^T  P\begin{pmatrix}
\nabla_x (\partial_{p_j} h)\\\nabla_p(\partial_{p_j} h)
\end{pmatrix} a_{ij}\right\}f_{\infty}dxdp\\+2\int_{ \mathbb{R}^{2d}} \begin{pmatrix}
\nabla_x h\\ \nabla_p h
\end{pmatrix}^T P \begin{pmatrix} 0\\ \sum_{i,j=1}^d \nabla_p a_{ij}\partial^2_{p_ip_j}h
\end{pmatrix} f_{\infty}dxdp\\ 
\leq  \int_{ \mathbb{R}^{2d}}\begin{pmatrix}
\nabla_x  h\\ \nabla_p h
\end{pmatrix}^T  
\begin{pmatrix}
\frac{\varepsilon^3 d}{\eta V_0^3 p_0^3}(I-\frac{p\otimes p}{p_0^2})& 0\\
0& \frac{2\varepsilon d}{\eta V_0 p_0} (I+p\otimes p)
\end{pmatrix} \begin{pmatrix}
\nabla_x  h\\ \nabla_p h
\end{pmatrix} f_{\infty}dxdp\\- \int_{ \mathbb{R}^{2d}} \varepsilon \left(2-\frac{3\eta}{p_0}\right)\frac{1}{V_0}w^T D \otimes D
 wf_{\infty}dxdp\\ \leq \int_{ \mathbb{R}^{2d}}\begin{pmatrix}
\nabla_x  h\\ \nabla_p h
\end{pmatrix}^T  
\begin{pmatrix}
\frac{\varepsilon^3 d}{\eta V^3_0 p_0^3}(I-\frac{p\otimes p}{p_0^2})& 0\\
0& \frac{2\varepsilon d}{\eta V_0 p_0} (I+p\otimes p)
\end{pmatrix} \begin{pmatrix}
\nabla_x  h\\ \nabla_p h
\end{pmatrix} f_{\infty}dxdp .
 \end{multline}
\subsubsection*{Step 2, Gr\"onwall type inequality:}
\eqref{dt E} and \eqref{26} show 
\begin{equation}\label{dt E1}
\frac{d }{dt}\mathrm{E}[h(t)]\leq -\frac{\lambda (2-\delta)}{2}\int_{ \mathbb{R}^{2d}} h^2 f_{\infty}dxdp 
-\int_{ \mathbb{R}^{2d}}  \begin{pmatrix}
\nabla_x h\\ \nabla_p h
\end{pmatrix}^T  P_1\begin{pmatrix}
\nabla_x h\\ \nabla_p h
\end{pmatrix} f_{\infty}dxdp,
\end{equation}
where
\begin{multline}\label{P_1 mat.}
P_1 \colonequals 2\gamma\begin{pmatrix}
0& 0\\
0& D
\end{pmatrix}+(QP+PQ^T)-\sum_{i=1}^d\left(\frac{p_i}{p_0} \partial_{x_i} P -\partial_{x_i} V  \partial_{p_i} P\right)-\sum_{i,j=1}^d \frac{1}{f_{\infty}}\partial_{p_j}(\partial_{p_i}Pa_{ij}f_{\infty})\\
-\begin{pmatrix}
\frac{\varepsilon^3 d}{\eta V_0^3 p_0^3}(I-\frac{p\otimes p}{p_0^2})& 0\\
0& \frac{2\varepsilon d}{\eta V_0 p_0} (I+p\otimes p)
\end{pmatrix}.
\end{multline}
The first two terms can be rewritten as 
\begin{multline*}
2\gamma \begin{pmatrix}
0& 0\\
0& D
\end{pmatrix}+(QP+PQ^T)\\
=\begin{pmatrix}
\frac{2\varepsilon^2}{V_0^2p_0^3}(I- \frac{p\otimes p}{p_0^2})& -\frac{\varepsilon^2}{V_0^2p_0^3}(I- \frac{p\otimes p}{p_0^2})(\frac{2\varepsilon}{V_0} \frac{\partial^2 V}{\partial x^2}+d I)+(\frac{\varepsilon^2}{V_0^2p_0^2}+\frac{2\varepsilon}{V_0p_0^2})I\\
- \frac{\varepsilon^2}{V_0^2p_0^3}(\frac{2\varepsilon}{V_0} \frac{\partial^2 V}{\partial x^2}+d I)(I- \frac{p\otimes p}{p_0^2})+(\frac{\varepsilon^2}{V_0^2p_0^2}+\frac{2\varepsilon}{V_0p_0^2})I& -\frac{2\varepsilon^2}{V_0^2p_0^2}\frac{\partial^2 V}{\partial x^2}+(\frac{4\varepsilon}{V_0}+2\gamma)\frac{I+p\otimes p}{p_0}-\frac{4\varepsilon d}{V_0p_0^2}I
\end{pmatrix}.
\end{multline*}
Then, using Lemma \ref{heavy comp} (from Appendix) for the forth term of \eqref{P_1 mat.} and Lemma \ref{heavy comp2} (from Appendix) for the third term of \eqref{P_1 mat.}, we obtain that there exist constants $\theta_1,\ \theta_2,\, \theta_3,\,\theta_4>0$ such that 
\begin{equation}\label{P_1>}
P_1\geq \begin{pmatrix}
X& Y^T\\
Y& Z
\end{pmatrix},
\end{equation} 
where $$X\colonequals \left(1-\theta_3 \varepsilon-\frac{\theta_1 \varepsilon}{V_0}-\frac{\varepsilon d}{2\eta V_0}\right)\frac{2\varepsilon^2}{V_0^2p_0^3}\left(I- \frac{p\otimes p}{p_0^2}\right),$$
 $$Y\colonequals - \frac{\varepsilon^2}{V_0^2p_0^3}\left(\frac{2\varepsilon}{V_0} \frac{\partial^2 V}{\partial x^2}+d I\right)\left(I- \frac{p\otimes p}{p_0^2}\right)+\left(\frac{\varepsilon}{V_0}+2\right)\frac{\varepsilon}{V_0p_0^2}I,$$
  $$Z\colonequals  -\frac{2\varepsilon^2}{V_0^2p_0^2}\frac{\partial^2 V}{\partial x^2}+\left(\frac{(4-2\theta_2)\varepsilon}{V_0}+2\gamma-2\theta_4 \varepsilon-\frac{2\varepsilon d}{\eta V_0 } \right)\frac{I+p\otimes p}{p_0}-\frac{4\varepsilon d}{V_0p_0^2}I
.$$
We choose a sufficiently small $\varepsilon>0$ such that 
$$1-\theta_3 \varepsilon-\frac{\theta_1 \varepsilon}{V_0(x)}-\frac{\varepsilon d}{2\eta V_0(x)}>\frac{1}{2}$$
for all $x\in \mathbb{R}^d.$  It is possible to choose such $\varepsilon>0$ because $\frac{1}{V_0(x)}$ is uniformly  bounded for $x\in \mathbb{R}^d.$
Then we have
\begin{equation}\label{X>}
X\geq \frac{\varepsilon^2}{V_0^2p_0^3}\left(I- \frac{p\otimes p}{p_0^2}\right)>0
\end{equation}
for all $ x, p \in \mathbb{R}^d.$ 
Since the elements of the matrix $\frac{1}{V_0} \frac{\partial^2 V}{\partial x^2}$ are bounded (due to Assumption \eqref{cond.V}) and $\frac{1}{p_0^3}\left(I- \frac{p\otimes p}{p_0^2}\right)\leq \frac{1}{p_0^2}I,$ if we (possibly)  choose  $\varepsilon>0$ even  smaller,  we have 
\begin{equation}\label{Y>} \frac{\varepsilon}{V_0p_0^2}I\leq Y\leq  \frac{3\varepsilon}{V_0p_0^2}I
\end{equation}
for all $ x, p \in \mathbb{R}^d.$ 
Similarly, since the elements of the matrix $\frac{1}{V_0} \frac{\partial^2 V}{\partial x^2}$ are bounded and $\frac{1}{p_0^2}\leq \frac{1}{p_0} (I+p\otimes p),$ if we (possibly)  choose $\varepsilon>0$ even smaller, we have 
\begin{equation}\label{Z>}
Z\geq \frac{2\gamma-1}{p_0} (I+p\otimes p)>0
\end{equation}
for all $ x, p \in \mathbb{R}^d.$ 
\eqref{X>}, \eqref{Y>}, and \eqref{Z>} show that, if $\varepsilon>0$ is small enough and $\gamma$ is large enough, then $\begin{pmatrix}
X& Y^T\\
Y& Z
\end{pmatrix}$ is positive definite and  there is a constant $C>0$ such that 
\begin{equation}\label{>CP}
\begin{pmatrix}
X& Y^T\\
Y& Z
\end{pmatrix}\geq CP.
\end{equation} 
We fix $ \varepsilon>0$ and $\gamma>0$ such that \eqref{X>}-\eqref{>CP} hold. 
Then \eqref{dt E1}, \eqref{P_1>}, and \eqref{>CP} imply
\begin{equation*}\label{dt E2}
\frac{d }{dt}\mathrm{E}[h(t)]\leq -\frac{\lambda (2-\delta)}{2}\int_{ \mathbb{R}^{2d}} h^2 f_{\infty}dxdp 
-C\int_{ \mathbb{R}^{2d}}  \begin{pmatrix}
\nabla_x h\\ \nabla_p h
\end{pmatrix}^T  P\begin{pmatrix}
\nabla_x h\\ \nabla_p h
\end{pmatrix} f_{\infty}dxdp.
\end{equation*}
 $\int_{ \mathbb{R}^{2d}} h^2 f_{\infty}dxdp $ and $\mathrm{H}_{\delta}[h]$ are equivalent by Theorem \ref{DMS}. Hence, from the equation above we conclude that there is a constant $\Lambda>0$ such that \eqref{dt E<-labmda E} holds. 
\end{proof}
\begin{Remark}
    Let us explain why we choose the matrix $P$ as in \eqref{Pmat} and comment on  the proof of Theorem \ref{th:hypocoercivity}. Our goal is to choose a positive definite matrix $P$ so that the Gr\"onwall inequality \eqref{dt E<-labmda E} holds. To get such inequality, the estimate in \eqref{dt E} shows that at least we need to construct a positive definite matrix $P$ such that 
    \begin{equation}\label{2gamma}
    2\gamma\begin{pmatrix}
0& 0\\
0& D
\end{pmatrix}+QP+PQ^T
\end{equation}
is positive definite  and bigger than a positive constant times $P.$  We constructed the matrix $P=P(x,p,\varepsilon)$ in \eqref{Pmat} so that this condition holds for sufficiently large $\gamma>0$ and sufficiently small $\varepsilon>0$ (for a similar construction, see \cite[Section 4]{AT}). The dependence of $P$ (as well as the norm in \eqref{H^1-norm}) on the matrix  $$I-\frac{p\otimes p}{p_0^2}$$
and its inverse
$$ I+p\otimes p$$
follows from the dependence of $Q$ on the matrix $I-\frac{p\otimes p}{p_0^2}$ (for the expression of $Q,$ see Lemma \ref{lemma main}). The proof of Theorem \ref{th:hypocoercivity} shows that  this construction of $P$  is actually enough to prove  the Gr\"onwall inequality \eqref{dt E<-labmda E}. The reason is that the term containing
$$  2\gamma\begin{pmatrix}
0& 0\\
0& D
\end{pmatrix}+(QP+PQ^T) $$  is bigger than the remaining terms in \eqref{dt E} for sufficiently large $\gamma>0$ and sufficiently small $\varepsilon>0.$
\end{Remark}

We are now ready to prove Theorem \ref{decay in H^1}.
\begin{proof}[\textbf{Proof of Theorem \ref{decay in H^1}}]
 $\mathrm{H}_{\delta}$ is equivalent to the square of the $L^2-$norm by Theorem\,\ref{DMS}\,$(i).$ This fact and  the inequalities \eqref{.<P<.}  show that $\mathrm{E}$ is equivalent to the $\mathscr{H}^1-$norm. Then the proof follows from  \eqref{Cor:H^1 decay}. 
\end{proof}
\section{Hypoelliptic regularity}\label{sec:hypoelliptic}
In this section, we prove Theorem \ref{th:hypoellipticity in H^1}, i.e., we show that, for any initial data $h_0 \in  L^2(\mathbb{R}^{2d}, f_{\infty}),$ the solution  $h(t)$ of \eqref{Eqh}  is in $\mathscr{H}^1(\mathbb{R}^{2d}, f_{\infty})$ for all $t > 0.$ Then, we shall prove Corollary \ref{corollary}.

\begin{proof}[\textbf{Proof of Theorem \ref{th:hypoellipticity in H^1}}] Let $\displaystyle h\colonequals \frac{f-f_{\infty}}{f_{\infty}}.$ Then $h$ solves \eqref{Eqh}.
  We define the functional 
$$\mathcal{E}[h]\colonequals \gamma ||h||^2_{L^2(\mathbb{R}^{2d}, f_{\infty})}+\mathrm{S}_P[h], $$ where  $\gamma>0,$ and $\mathrm{S}_P[h]$ is defined in \eqref{S_P}. 
In order to prove the short-time regularization of \eqref{hypel11} and \eqref{hypoel21} we consider this functional with  a matrix $P$ which depends not only on $x$ and  $p$ but also on time $t,$ i.e. 
\begin{equation*}
P=P(t, x,p)\colonequals \begin{pmatrix}
\frac{2\varepsilon^3t^3}{V_0^3p_0^3}(I- \frac{p\otimes p}{p_0^2})& \frac{\varepsilon^2t^2}{V_0^2p_0^2}I\\
\frac{\varepsilon^2t^2}{V_0^2p_0^2} I& \frac{2\varepsilon t}{V_0 p_0} (I+p\otimes p)
\end{pmatrix},
\end{equation*}
 where $\varepsilon>0$ will be fixed later. Compared to \eqref{Pmat}, $\varepsilon$ was replaced by $\varepsilon t.$ It is easy to check 
 \begin{equation}\label{.<P<.1}\begin{pmatrix}
\frac{ \varepsilon^3 t^3}{V_0^3p_0^{3}} (I-\frac{p\otimes p}{p_0^2})&0\\0&   \frac{\varepsilon t}{V_0 p_0} (I+p\otimes p)
\end{pmatrix}\leq P,
\end{equation}which implies that $P(t,x, p)$ is positive definite for all $t>0$ and $x, \,p\in \mathbb{R}^d.$  Our goal is to show that $\mathcal{E}[h(t)]$  decreases. To this end we compute the time derivative of $\mathcal{E}[h(t)].$  We follow the proofs of Lemma \ref{lemma ||h||^2} and Lemma \ref{lemma main} to compute the time derivative of $\mathcal{E}[h(t)],$ but now we need to take into account that $P$ depends on time $t:$ 
\begin{multline}\label{dt Et}
\frac{d }{dt}\mathcal{E}[h(t)]= -2\gamma\int_{ \mathbb{R}^{2d}}\nabla^T_p hD \nabla_p hf_{\infty}dxdp\\
-2\int_{ \mathbb{R}^{2d}}\left\{\sum_{i, j=1}^d \begin{pmatrix}
\nabla_x ( \partial_{p_i} h)\\ \nabla_p(\partial_{p_i} h)
\end{pmatrix}^T  P\begin{pmatrix}
\nabla_x (\partial_{p_j} h)\\\nabla_p(\partial_{p_j} h)
\end{pmatrix} a_{ij}\right\}f_{\infty}dxdp\\+2\int_{ \mathbb{R}^{2d}}  \begin{pmatrix}
\nabla_x h\\ \nabla_p h
\end{pmatrix}^T P \begin{pmatrix} 0\\\sum_{i,j=1}^d \nabla_p a_{ij}\partial^2_{p_ip_j}h
\end{pmatrix} f_{\infty}dxdp\\
-\int_{ \mathbb{R}^{2d}}  \begin{pmatrix}
\nabla_x h\\ \nabla_p h
\end{pmatrix}^T  \left\{QP+PQ^T-\partial_t P\right\}\begin{pmatrix}
\nabla_x h\\ \nabla_p h
\end{pmatrix} f_{\infty}dxdp\\
+\int_{\mathbb{R}^{2d}}  \begin{pmatrix}
\nabla_x h\\ \nabla_p h
\end{pmatrix}^T  \left\{\sum_{i=1}^d\left(\frac{p_i}{p_0} \partial_{x_i} P -\partial_{x_i} V  \partial_{p_i} P\right)+\sum_{i,j=1}^d \frac{1}{f_{\infty}}\partial_{p_j}(\partial_{p_i}Pa_{ij}f_{\infty})\right\}\begin{pmatrix}
\nabla_x h\\ \nabla_p h
\end{pmatrix} f_{\infty}dxdp.
\end{multline}
We estimate the terms on the right as in \eqref{22}-\eqref{26} (where we need to replace $\varepsilon$ by $\varepsilon t$) and obtain
\begin{equation}\label{dt Et1}
\frac{d }{dt}\mathcal{E}[h(t)]\leq  
-\int_{ \mathbb{R}^{2d}}  \begin{pmatrix}
\nabla_x h\\ \nabla_p h
\end{pmatrix}^T  P_2\begin{pmatrix}
\nabla_x h\\ \nabla_p h
\end{pmatrix} f_{\infty}dxdp,
\end{equation}
where
\begin{multline}\label{P_2 mat.}
P_2 \colonequals 2\gamma\begin{pmatrix}
0& 0\\
0& D
\end{pmatrix}+(QP+PQ^T)-\partial_t P-\sum_{i=1}^d\left(\frac{p_i}{p_0} \partial_{x_i} P -\partial_{x_i} V  \partial_{p_i} P\right)-\sum_{i,j=1}^d \frac{1}{f_{\infty}}\partial_{p_j}(\partial_{p_i}Pa_{ij}f_{\infty})\\
-\begin{pmatrix}
\frac{\varepsilon^3t^3 d}{\eta V_0^3p_0^3}(I-\frac{p\otimes p}{p_0^2})& 0\\
0& \frac{2\varepsilon t d}{\eta V_0 p_0} (I+p\otimes p)
\end{pmatrix};
\end{multline}
note the similarity of $P_2$ with $P_1$ in \eqref{P_1 mat.}.
The first three terms can be rewritten as 
\begin{multline*}
2\gamma\begin{pmatrix}
0& 0\\
0& D
\end{pmatrix}+(QP+PQ^T)-\partial_t P\\
=\begin{pmatrix}
(1-\frac{3\varepsilon }{V_0})\frac{2\varepsilon^2t^2}{V_0^2p_0^3}(I- \frac{p\otimes p}{p_0^2})& -\frac{\varepsilon^2t^2}{V_0^2p_0^3}(I- \frac{p\otimes p}{p_0^2})(\frac{2\varepsilon t}{V_0} \frac{\partial^2 V}{\partial x^2}+d I)+(\frac{\varepsilon t-2\varepsilon}{V_0}+2)\frac{\varepsilon t}{V_0p_0^2}I\\
- \frac{\varepsilon^2 t^2}{V_0^2p_0^3}(\frac{2\varepsilon t}{V_0} \frac{\partial^2 V}{\partial x^2}+d I)(I- \frac{p\otimes p}{p_0^2})+(\frac{\varepsilon t-2\varepsilon}{V_0}+2)\frac{\varepsilon t}{V_0p_0^2}I& -\frac{2\varepsilon^2t^2}{V_0^2p_0^2}\frac{\partial^2 V}{\partial x^2}+(\frac{4\varepsilon t-2 \varepsilon}{V_0}+2\gamma)\frac{I+p\otimes p}{p_0}-\frac{4\varepsilon t d}{V_0p_0^2}I
\end{pmatrix},
\end{multline*}
Then, using (from Appendix) Lemma \ref{heavy comp} for the fifth term of \eqref{P_2 mat.} and Lemma \ref{heavy comp2} for the forth term of \eqref{P_2 mat.} (where we need to replace $\varepsilon$ by $\varepsilon t$), we obtain that there exist constants $\theta_1,\,\theta_2,\,\theta_3\,\theta_4>0$ such that
\begin{equation}\label{tP_1>}
P_2\geq \begin{pmatrix}
X& Y^T\\
Y& Z
\end{pmatrix},
\end{equation} 
where $$X\colonequals \left(1-\frac{3\varepsilon }{V_0}-\theta_3 \varepsilon t-\frac{\theta_1 \varepsilon t}{V_0}-\frac{\varepsilon t d}{2\eta V_0}\right)\frac{2\varepsilon^2 t^2}{V_0^2p_0^3}\left(I- \frac{p\otimes p}{p_0^2}\right),$$
 $$Y\colonequals - \frac{\varepsilon^2t^2}{V_0^2p_0^3}\left(\frac{2\varepsilon t}{V_0} \frac{\partial^2 V}{\partial x^2}+d I\right)\left(I- \frac{p\otimes p}{p_0^2}\right)+\left(\frac{\varepsilon t-2\varepsilon}{V_0}+2\right)\frac{\varepsilon t}{V_0p_0^2}I,$$
  $$Z\colonequals  -\frac{2\varepsilon^2t^2}{V_0^2p_0^2}\frac{\partial^2 V}{\partial x^2}+\left(\frac{(4-2\theta_2)\varepsilon t-2\varepsilon}{V_0}+2\gamma-2\theta_4 \varepsilon t -\frac{2\varepsilon t d}{\eta V_0 }\right)\frac{I+p\otimes p}{p_0}-\frac{4\varepsilon t d}{V_0p_0^2}I
.$$
We choose a sufficiently small $\varepsilon>0$ such that 
$$1-\frac{3\varepsilon }{V_0}-\theta_3 \varepsilon t-\frac{\theta_1 \varepsilon t}{V_0}-\frac{\varepsilon t d}{2\eta V_0}>\frac{1}{2}$$
for all $x\in \mathbb{R}^d,$ $t\in [0,t_0].$  It is possible to choose such $\varepsilon>0$ because $\frac{1}{V_0(x)}$ is bounded for  $x\in \mathbb{R}^d$ and $t$ varies in a bounded interval.
Then we have
\begin{equation}\label{tX>}
X\geq \frac{\varepsilon^2 t^2}{V_0^2p_0^3}\left(I- \frac{p\otimes p}{p_0^2}\right)\geq 0
\end{equation}
for all $ x, p \in \mathbb{R}^d,$ $t\in [0, t_0].$ 
Since the elements of the matrix $\frac{1}{V_0} \frac{\partial^2 V}{\partial x^2}$ are uniformly bounded by Assumption \eqref{cond.V} and $\frac{1}{p_0^3}\left(I- \frac{p\otimes p}{p_0^2}\right)\leq \frac{1}{p_0^2}I,$ if we (possibly) choose  $\varepsilon>0$ even smaller, then we have 
\begin{equation}\label{tY>} \frac{\varepsilon t}{V_0p_0^2}I\leq Y\leq  \frac{3 \varepsilon t}{V_0p_0^2}I
\end{equation}
for all $ x, p \in \mathbb{R}^d,$ $t\in [0, t_0].$ 
Similarly, since the elements of the matrix $\frac{1}{V_0} \frac{\partial^2 V}{\partial x^2}$ are bounded and $\frac{1}{p_0^2}\leq \frac{1}{p_0} (I+p\otimes p),$ if we (possible) choose $\varepsilon>0$ even smaller,  we have
\begin{equation}\label{tZ>}
Z\geq \frac{2\gamma-1}{p_0} (I+p\otimes p)\geq 0
\end{equation}
for all $ x, p \in \mathbb{R}^d,$ $t\in [0, t_0].$
\eqref{tX>}, \eqref{tY>}, and \eqref{tZ>} show that, if $ \varepsilon>0$ is small enough and $\gamma>0$ is large enough, then 
\begin{equation}\label{tXYZ}
\begin{pmatrix}
X& Y^T\\
Y& Z
\end{pmatrix}\geq 0.
\end{equation} 
We fix $ \varepsilon>0$ and $\gamma>0$  such that \eqref{tX>}-\eqref{tXYZ} hold.
Then \eqref{dt Et1},  \eqref{tP_1>}, and \eqref{tXYZ} imply
\begin{equation*}
\frac{d }{dt}\mathcal{E}[h(t)]\leq 0.
\end{equation*}
 Hence,  $\mathcal{E}[h(t)]$ is decreasing in $[0,t_0]$ and therefore
\begin{equation}\label{E<gamma ||||} \mathcal{E}[h(t)]\leq \mathcal{E}[h(0)]= \gamma||h_0||^2_{L^2(\mathbb{R}^{2d}, f_{\infty})}, \, \, \,  \, \, \forall \,t \in [0,t_0].
\end{equation}
Moreover, we have by \eqref{.<P<.1} that
\begin{multline}\label{E>||.||}\mathcal{E}[h(t)] \geq \varepsilon^3t^3 \int_{ \mathbb{R}^{2d}}\frac{1}{V_0^3p_0^3}\nabla^T_x h(t)\left(I-\frac{p\otimes p}{p_0^2}\right)\nabla_x h(t)f_{\infty}dxdp\\+\varepsilon t \int_{\mathbb{R}^{2d}}\frac{1}{V_0p_0}\nabla^T_p h(t)(I+p\otimes p)\nabla_p h(t)f_{\infty}dxdp.  
 \end{multline} \eqref{E<gamma ||||} and \eqref{E>||.||} show that  
 $$ \int_{\mathbb{R}^{2d}}\frac{1}{V^3_0p_0^3}\nabla^T_x h(t)\left(I-\frac{p\otimes p}{p_0^2}\right)\nabla_x h(t)f_{\infty}dxdp\leq \frac{\gamma}{\varepsilon^3t^3}||h_0||^2_{L^2(\mathbb{R}^{2d}, f_{\infty})}$$
 and $$\int_{\mathbb{R}^{2d}}\frac{1}{V_0p_0}\nabla^T_p h(t)(I+p\otimes p)\nabla_p h(t)f_{\infty}dxdp\leq \frac{\gamma}{\varepsilon t}||h_0||^2_{L^2(\mathbb{R}^{2d}, f_{\infty})}.$$
Hence, \eqref{hypel11} and \eqref{hypoel21} hold with constants $C_3\colonequals \frac{\gamma}{\varepsilon^3}$ and $C_4\colonequals \frac{\gamma}{\varepsilon}.$ \eqref{hypoel3}  follows easily by adding these estimates.  
\end{proof}
\begin{proof}[\textbf{Proof of Corollary \ref{corollary}}]
Let $t_0>0.$  Theorem \ref{th:hypoellipticity in H^1}  and Theorem \ref{decay in H^1} show that $\frac{f(t_0)}{f_{\infty}}\in \mathscr{H}^1(\mathbb{R}^{2d}, f_{\infty})$ and 
\begin{equation*}
\left|\left|\frac{f(t)-f_{\infty}}{f_{\infty}}\right|\right|_{\mathscr{H}^1(\mathbb{R}^{2d}, f_{\infty})}\leq C_2 e^{-\Lambda (t-t_0)} \left|\left|\frac{f(t_0)-f_{\infty}}{f_{\infty}}\right|\right|_{\mathscr{H}^1(\mathbb{R}^{2d}, f_{\infty})}
\end{equation*}
holds for all $t\geq t_0>0$  with the constant $C_2$ and the rate $\Lambda$ given in Theorem \ref{decay in H^1}. Using \eqref{hypoel3} at $t=t_0,$ we get 
\begin{equation*}
\left|\left|\frac{f(t)-f_{\infty}}{f_{\infty}}\right|\right|_{\mathscr{H}^1(\mathbb{R}^{2d}, f_{\infty})}\leq   \frac{C_2(C_3+C_4 t_0^2)^{1/2} e^{\Lambda t_0}}{t_0^{3/2}} e^{-\Lambda t}  \left|\left|\frac{f_0-f_{\infty}}{f_{\infty}}\right|\right|_{L^2(\mathbb{R}^{2d}, f_{\infty})}.
\end{equation*}
This proves the claimed estimate with the constant $C_5 \colonequals  \frac{C_2(C_3+C_4 t_0^2)^{1/2} e^{\Lambda t_0}}{t_0^{3/2}}.$
\end{proof}

\section{Newtonian limit $c\to\infty$}\label{sec:limit}
%
So far we analyzed the relativistic Fokker-Planck equation \eqref{Eq0} only for the fixed constant $c=1$, showing that its solutions converge with an explicit exponential rate toward equilibrium $f_\infty$, see Theorem \ref{main result}. Since solutions of the classical kinetic Fokker-Planck equation \eqref{KFP} also converge exponentially to its equilibrium, it is natural to ask if the $L^2$-decay estimate \eqref{L2-decay} holds uniformly as $c\to\infty$. As this problem has not been solved yet, we shall give here some comments for the case of the space homogeneous equation \eqref{HFP} with
\begin{equation*}\label{matrixDc}
  D(p)\colonequals\frac{I+\frac{p\otimes p}{c^2}}{\sqrt{1+\frac{|p|^2}{c^2}}}\in \mathbb{R}^{d\times d},    
\end{equation*}
and $c>0$, but keeping $m=q=\sigma=\nu=1$. This amounts to proving uniformity in the decay estimate \eqref{p-decay} as $c\to\infty$. More precisely, we discuss here only how the decay rate $\kappa_1$ in \eqref{p-decay} (which also appears in the Poincar\'e inequality \eqref{Poincare M}) depends on $c.$

In \cite{Felix.Cal.}, the Bakry-Emery method \cite{Con.Sob, BGL} has been used to prove a log-Sobolev inequality for the probability density 
\begin{equation*}\label{Mc}
  M_c(p)\colonequals \frac{e^{-c\sqrt{c^2+|p|^2}}}{ \int_{\mathbb{R}^{d}}e^{-c\sqrt{c^2+|p'|^2}}dp'},
\end{equation*} 
with an explicit (but in general not sharp) Sobolev constant. 
Since the Bakry-Emery method also applies to weighted $L^2$-norms \cite{Con.Sob}, the same constant can be taken for the Poincar\'e inequality for $M_c(p)$.  
Theorem 2 in \cite{Felix.Cal.} yields the following $c$-dependent constant in the  log-Sobolev inequality:
$$
  \frac{\kappa_1}{2}= P_c\big(\frac{2 c^2+ c\sqrt{4 c^4-39}}{13}\big), \quad \mbox{for } c>2,
$$
where $P_c(u)=\frac{2 cu^3-13u^2+2 c^3u-c^2}{4cu^3}.$ We have $\lim_{c\to2} \frac{\kappa_1}{2}\approx 0.1812$. Moreover, one can check that $\kappa_1(c)$ increases monotonically and it converges to 1 as $c\to \infty.$ This is commensurate with the sharp decay rate $\kappa_1=1$  of the classical homogeneous Fokker-Planck equation (i.e. \eqref{HFP} with $D\equiv I$) in $L^2(\R^d,M_\infty)$ with the Gaussian equilibrium
$$
  M_\infty(p)\colonequals \frac{e^{-|p|^2/2}}{ \int_{\mathbb{R}^{d}}e^{-|p'|^2/2}dp'},
$$
see \cite{Con.Sob}. 
%

\section{Conclusion and Outlook}\label{sec:conclusion}
In this paper we established exponential convergence of solutions for the relativistic Fokker-Planck equation to the unique steady state, and we included cases with non-convex external potentials (e.g. double-well potentials). Our main results are:
\begin{enumerate}
\item[(a)] Exponential convergence in a weighted $L^2$ setting, by following the $L^2-$hypocoercivity method of Dolbeault-Mouhot-Schmeiser \cite{dolbeault2015hypocoercivity}.
\item[(b)] Exponential convergence in a weighted $H^1$ setting, by establishing a new, refined entropy method that is based on a modified Fisher-type functional with a metric that depends on position and momentum.
\item[(c)] Hypoelliptic regularization from $L^2$ to $H^1$ by making the Lyapunov functional from (b) depend also on time.
\end{enumerate}
In the proof of (a), we used the $L^2-$hypocoercivity method of Dolbeault-Mouhot-Schmeiser \cite{dolbeault2015hypocoercivity}. This method was applied to the classical kinetic Fokker-Planck equation in \cite{dolbeault2015hypocoercivity}. We showed that this method can be applied to the relativistic Fokker-Planck equation which, in some sense, is more general than the classical one. As an extension of the present work one could study the long-time behavior of general degenerate non-symmetric Fokker-Planck equations with non-constant diffusion matrices. Such degenerate extensions of the kinetic Fokker-Planck equation (but with constant diffusion matrices and linear drift) have been considered in \cite{FP1, FP2, FP3}.

In general, relativistic kinetic equations are mathematically more challenging than classical ones since the relativistic transport and collision terms are more intricate, thus complicating their mathematical structure. Our results show that the usual hypocoercive methods developed for the classical Fokker-Planck equation can be adapted to the relativistic case. 
We expect that the techniques used in this paper can be useful to study the long-time behavior and hypoelliptic regularization of other relativistic kinetic equations, as the (linearized) relativistic Landau equation, e.g.

\section{Appendix}
\subsection{Weighted Poincar\'e inequalities and an elliptic regularity result}\label{6.1}


In this section we consider the elliptic equation 
\begin{equation}\label{ellip.eq.gen}
u(x) -\frac{a}{\rho_{\infty}(x)}\text{div}_x( \nabla_x  u(x) \rho_{\infty}(x) )=w(x), \, \, \, \, \, \, \, x \in \mathbb{R}^d,
\end{equation}
where $u$ is unknown, $a$ is a positive constant, and $w$ is a given function.  We will establish some regularity estimates for this equation in $L^2(\mathbb{R}^d, \rho_{\infty})$ with $\rho_{\infty}>0$ from \eqref{star1}. 
\begin{theorem}\label{lem:ellip} Let $w \in L^2(\mathbb{R}^d, \rho_{\infty})$ and $\int_{\mathbb{R}^d} w \rho_{\infty}dx=0.$ Assume that the potential $V$ satifies Assumption \ref{Assumptions i-ii}. 
Then, there are positive constants $C_1$ and $C_2$ such that the unique solution of \eqref{ellip.eq.gen} satisfies
\begin{equation}\label{C_1}
\int_{\mathbb{R}^d}|\nabla_x u|^2 |\nabla_x V|^2 \rho_{\infty}dx\leq C_1\int_{\mathbb{R}^d} w^2 \rho_{\infty}dx,
\end{equation}
\begin{equation}\label{C_2}
\int_{\mathbb{R}^d}\left| \left|\frac{\partial^2 u}{\partial x^2} \right|\right|_F^2 \rho_{\infty}dx \leq C_2\int_{\mathbb{R}^d} w^2 \rho_{\infty}dx.
\end{equation}
\end{theorem}
  To prove Theorem \ref{lem:ellip}, we  need  the weighted Poincar\'e inequalities   \eqref{weighted poin.1} and \eqref{weighted poin.2} below. We mention that these inequalities were obtained in \cite{dolbeault2015hypocoercivity} in a  general setting, but we provide proofs  for being self-contained.
\begin{lemma}\label{lemma. weight.poin}
 Assume that Assumption \ref{Assumptions i-ii}  holds.
 Then
 \begin{itemize}
 \item[i)]
There exists $\kappa_3>0$ such that 
\begin{equation}\label{weighted poin.1}
\int_{\mathbb{R}^d}h^2|\nabla_x V|^2\rho_{\infty}dx\leq \frac{1}{\kappa_3}\int_{\mathbb{R}^d}|\nabla_xh|^2 \rho_{\infty}dx
\end{equation}
holds for all $ h\in L^2(\mathbb{R}^d, \rho_{\infty})$ with $ |\nabla_x h|\in L^2(\mathbb{R}^d, \rho_{\infty})$ and  $\int_{\mathbb{R}^d} h \rho_{\infty}dx=0.$
\item[ii)] There exists $\kappa_4>0$ such that 
\begin{equation}\label{weighted poin.2}
\int_{\mathbb{R}^d}h^2(1+|\nabla_x V|^2)|\nabla_x V|^2\rho_{\infty}dx\leq \frac{1}{\kappa_4}\int_{\mathbb{R}^d}|\nabla_x h|^2 (1+|\nabla_x V|^2) \rho_{\infty}dx
\end{equation}
holds for all $ h\in L^2(\mathbb{R}^d, \rho_{\infty})$ with $ |\nabla_x h|(1+|\nabla_x V|)\in L^2(\mathbb{R}^d, \rho_{\infty})$ and $\int_{\mathbb{R}^d} h \rho_{\infty}dx=0.$
\end{itemize}  
\end{lemma}
\begin{proof} $i)$
By the identity $\sqrt{\rho_{\infty} }\,\nabla_x h=\nabla_x (h\sqrt{\rho_{\infty}} )+\frac{h\sqrt{\rho_{\infty}}}{2}\nabla_x V$ and integrating by parts
\begin{align*}\int_{\mathbb{R}^d}|\nabla_x h|^2 \rho_{\infty}dx &\geq \frac{1}{4}\int_{\mathbb{R}^d}h^2|\nabla_x V|^2\rho_{\infty}dx+\int_{\mathbb{R}^d}h\sqrt{\rho_{\infty}} \,\nabla_x (h\sqrt{\rho_{\infty} })\cdot\nabla_x Vdx\\&= \frac{1}{4}\int_{\mathbb{R}^d}h^2|\nabla_x V
|^2\rho_{\infty}dx-\frac{1}{2}\int_{\mathbb{R}^d}h^2 \Delta_x V \rho_{\infty}dx.
\end{align*} 
This estimate and the first condition in \eqref{cond.V} show 
\begin{equation*}
\int_{\mathbb{R}^d}|\nabla_x h|^2 \rho_{\infty}dx\geq \frac{1-c_2}{4}\int_{\mathbb{R}^d}h^2|\nabla_x V
|^2\rho_{\infty}dx-\frac{c_1}{2}\int_{\mathbb{R}^d}h^2\rho_{\infty} dx.
\end{equation*}
Then, \eqref{Poincare rho} lets us  obtain \eqref{weighted poin.1} with the constant $\kappa_3\colonequals \frac{(1-c_2)(c_1+2\kappa_2)}{8\kappa_2}.$ \\

$ii)$ We recall  $V_0(x) \colonequals \sqrt{1+|\nabla_x V(x)|^2}.$ Let $\bar{h}\colonequals\int_{\mathbb{R}^d}h V_0\rho_{\infty}dx,$ then by \eqref{weighted poin.1}
 \begin{equation*}
\int_{\mathbb{R}^d}(hV_0-\bar{h})^2|\nabla_x V|^2\rho_{\infty}dx\leq \frac{1}{\kappa_3}\int_{\mathbb{R}^d}\left|\nabla_x(hV_0)\right|^2 \rho_{\infty}dx.
\end{equation*} 
This leads to
 \begin{equation}\label{grad0}
\int_{\mathbb{R}^d}h^2V_0^2|\nabla_x V|^2\rho_{\infty}dx\leq  \frac{1}{\kappa_3}\int_{\mathbb{R}^d}\left|\nabla_x(hV_0)\right|^2 \rho_{\infty}dx+2\bar{h}\int_{\mathbb{R}^d}hV_0|\nabla_x V|^2\rho_{\infty}dx.
\end{equation} 
Next, we estimate the terms on the right hand side of \eqref{grad0}:
\begin{align}\label{grad1}
\int_{\mathbb{R}^d}\left|\nabla_x(hV_0)\right|^2 \rho_{\infty}dx&=\int_{\mathbb{R}^d}\left|\nabla_xhV_0+{h}\frac{\partial^2 V}{\partial x^2}\frac{\nabla_x V}{V_0}\right|^2 \rho_{\infty}dx\nonumber \\ & \leq 
2\int_{\mathbb{R}^d}|\nabla_x h|^2V_0^2 \rho_{\infty}dx+2\int_{\mathbb{R}^d} h^2\left|\left|\frac{\partial^2 V}{\partial x^2}\right|\right|_F^2\frac{|\nabla_x V|^2}{V_0^2} \rho_{\infty}dx\nonumber \\
 & \leq 
2\int_{\mathbb{R}^d}|\nabla_x h|^2V_0^2 \rho_{\infty}dx+4c_3^2\int_{\mathbb{R}^d} h^2|\nabla_x V|^2 \rho_{\infty}dx\nonumber \\ & \leq  
2\int_{\mathbb{R}^d}|\nabla_x h|^2V_0^2 \rho_{\infty}dx+\frac{4c_3^2}{\kappa_3}\int_{\mathbb{R}^d} |\nabla_x h|^2\rho_{\infty}dx, 
\end{align} 
where we used the second condition in \eqref{cond.V} and \eqref{weighted poin.1}. By the H\"older inequality and \eqref{Poincare rho} 
  $$|\bar{h}|\leq ||V_0||_{L^2(\mathbb{R}^d, \rho_{\infty})}||h||_{L^2(\mathbb{R}^d, \rho_{\infty})}  \leq \frac{1}{\sqrt{\kappa_2}} ||V_0||_{L^2(\mathbb{R}^d, \rho_{\infty})}  ||\nabla_x h||_{L^2(\mathbb{R}^d, \rho_{\infty})}. $$  We note that $||V_0||_{L^2(\mathbb{R}^d, \rho_{\infty})}$ is finite, because  \eqref{star1} and the first condition in \eqref{cond.V} yields:  
$$\int_{\mathbb{R}^d}|\nabla_x V|^2\rho_{\infty}dx=\int_{\mathbb{R}^d} \Delta_x V \rho_{\infty}dx\leq c_1+\frac{c_2}{2}\int_{\mathbb{R}^d}|\nabla_x V|^2\rho_{\infty}dx,$$
hence  $\int_{\mathbb{R}^d}|\nabla_x V|^2\rho_{\infty}dx\leq \frac{2c_1}{2-c_2}.$
Then, the H\"older inequality shows
 \begin{equation}\label{grad2}
2\bar{h}\int_{\mathbb{R}^d}hV_0|\nabla_x V|^2\rho_{\infty}dx\leq \frac{2  ||V_0||^4_{L^2(\mathbb{R}^d, \rho_{\infty})}}{\kappa_2}\int_{\mathbb{R}^d}|\nabla_x h|^2\rho_{\infty}dx+ 
\frac{1}{2}{\int_{\mathbb{R}^d}h^2V_0^2|\nabla_x V|^2\rho_{\infty}dx}.
\end{equation}
\eqref{grad0}, \eqref{grad1}, and \eqref{grad2} yield 
\begin{equation*}
\int_{\mathbb{R}^d}h^2V_0^2|\nabla_x V|^2\rho_{\infty}dx\leq \int_{\mathbb{R}^d}(4\kappa_2^{-1}  ||V_0||^4_{L^2(\mathbb{R}^d, \rho_{\infty})}+8c_3^2\kappa_3^{-1}+4V_0^2)|\nabla_x h|^2\rho_{\infty}dx.
\end{equation*}
Therefore, we obtain \eqref{weighted poin.2} with $\kappa_4^{-1} \colonequals 4\kappa_2^{-1}  ||V_0||^4_{L^2(\mathbb{R}^d, \rho_{\infty})}+8c_3^2\kappa_3^{-1}+4.$

\end{proof}

\begin{proof}[\textbf{Proof of Theorem \ref{lem:ellip}}] Multiplying \eqref{ellip.eq.gen} by $\rho_{\infty}$ and integrating by parts we obtain $$\int_{\mathbb{R}^d} u \rho_{\infty}dx=\int_{\mathbb{R}^d} w \rho_{\infty}dx=0.$$
Next we multiply \eqref{ellip.eq.gen} by $u \rho_{\infty},$  integrate by parts, and use the H\"older inequality:
\begin{equation*}
\int_{\mathbb{R}^d}u^2\rho_{\infty}dx+a\int_{\mathbb{R}^d}|  \nabla_x u|^2\rho_{\infty}dx=\int_{\mathbb{R}^d} uw\rho_{\infty}dx\leq \frac{1}{2} \int_{\mathbb{R}^d}w^2\rho_{\infty}dx+\frac{1}{2} \int_{\mathbb{R}^d}u^2\rho_{\infty}dx.
\end{equation*}
Then
\begin{equation}\label{grad varphi}
\int_{\mathbb{R}^d}u^2 \rho_{\infty}dx+2a\int_{\mathbb{R}^d}|\nabla_x u|^2\rho_{\infty}dx
\leq \int_{\mathbb{R}^d}w^2\rho_{\infty}dx.
\end{equation}
We start   proving \eqref{C_1}: We multiply \eqref{ellip.eq.gen} by $u|\nabla_x V|^2 \rho_{\infty}$ and integrate by parts
\begin{align}\label{varphi V}
\int_{\mathbb{R}^d}u^2|\nabla_xV|^2 \rho_{\infty} dx& +a\int_{\mathbb{R}^d} | \nabla_x u|^2| \nabla_x V|^2 \rho_{\infty} dx \nonumber \\
&=\int_{\mathbb{R}^d} w u|\nabla_x V|^2 \rho_{\infty} dx- a\int_{\mathbb{R}^d} u \nabla_x u \cdot  \nabla_x (| \nabla_x V|^2) \rho_{\infty} dx.
\end{align}
Using the H\"older inequality we estimate the  terms on the right hand side of \eqref{varphi V} 
\begin{equation}\label{varphi V 1}
\int_{\mathbb{R}^d} w u|\nabla_x V|^2 \rho_{\infty} dx \leq \frac{1}{2\delta} 
 \int_{\mathbb{R}^d} w^2  \rho_{\infty} dx+\frac{\delta}{2}{\int_{\mathbb{R}^d}u^2|\nabla_x V|^4 \rho_{\infty} dx}
\end{equation}
and 
\begin{align}\label{varphi V 2}
- a\int_{\mathbb{R}^d} u \nabla_x u  \cdot \nabla_x (| \nabla_x V|^2) \rho_{\infty} dx &=-2a \int_{\mathbb{R}^d} u\nabla_x u \cdot \left(\frac{\partial^2 V}{\partial x^2} \nabla_x V\right) \rho_{\infty}dx
\nonumber \\
& \leq 2 a\int_{\mathbb{R}^d} |u||\nabla_x u| \left|\left|\frac{\partial^2 V}{\partial x^2}\right|\right|_F| \nabla_x V| \rho_{\infty}dx \nonumber \\
& \leq \varepsilon \int_{\mathbb{R}^d} |\nabla_x u|^2 | \nabla_x V |^2 \rho_{\infty} dx
+\frac{a^2}{\varepsilon}  \int_{\mathbb{R}^d} u^2\left|\left|  \frac{\partial^2 V}{\partial x^2}\right|\right|_F^2\rho_{\infty} dx,
\end{align}
for any $\delta>0, \, \, \,  \varepsilon>0.$  
\eqref{varphi V}, \eqref{varphi V 1}, and \eqref{varphi V 2} show that 
\begin{multline}\label{varphi V3}
\int_{\mathbb{R}^d}u^2|\nabla_xV|^2 \rho_{\infty} dx+(a-\varepsilon )\int_{\mathbb{R}^d}  |\nabla_x u|^2 | \nabla_x V|^2 \rho_{\infty} dx\\
\leq \frac{1}{2\delta} 
 \int_{\mathbb{R}^d} w^2  \rho_{\infty} dx+\frac{\delta}{2}{\int_{\mathbb{R}^d}u^2|\nabla_x V|^4 \rho_{\infty} dx}
+\frac{a^2}{\varepsilon}  \int_{\mathbb{R}^d} u^2\left|\left|  \frac{\partial^2 V}{\partial x^2}\right|\right|_F^2\rho_{\infty} dx.
\end{multline}
The Poincar\'e inequality \eqref{weighted poin.2} and \eqref{grad varphi} imply 
\begin{align}\label{varphi V4}
\int_{\mathbb{R}^d}u^2|\nabla_x V|^4 \rho_{\infty} dx & \leq \frac{1}{ \kappa_4}\int_{\mathbb{R}^d}| \nabla_x u|^2 (1+|\nabla_x V( x)|^2) \rho_{\infty} dx \nonumber \\
& \leq \frac{1}{2a\kappa_4}\int_{\mathbb{R}^d} w^2  \rho_{\infty} dx+\frac{1}{\kappa_4} \int_{\mathbb{R}^d}|\nabla_x u|^2 |\nabla_x V|^2 \rho_{\infty}dx.
\end{align}
 To estimate the last term in \eqref{varphi V3}, we use the second condition in \eqref{cond.V}, \eqref{grad varphi}, and the Poincar\'e inequality \eqref{weighted poin.1} 
\begin{align}\label{varphi V5}
 \int_{\mathbb{R}^d} u^2\left|\left|  \frac{\partial^2 V}{\partial x^2}\right|\right|_F^2\rho_{\infty} dx
& \leq 2 c^2_3\int_{\mathbb{R}^d} u^2(1+|\nabla_xV|^2)\rho_{\infty} dx 
\nonumber \\ &\leq 2 c^2_3\left(\int_{\mathbb{R}^d} w^2\rho_{\infty} dx+\frac{1}{\kappa_3}  \int_{\mathbb{R}^d} |\nabla_x u|^2\rho_{\infty} dx\right)\nonumber \\ &
\leq 2 c^2_3\left(1+\frac{1}{2a\kappa_3}\right)\int_{\mathbb{R}^d} w^2\rho_{\infty} dx. 
\end{align}
\eqref{varphi V3}, \eqref{varphi V4}, and \eqref{varphi V5} show that 
\begin{multline}\label{varphi V6}
\int_{\mathbb{R}^d}u^2|\nabla_xV|^2 \rho_{\infty} dx+(a-\varepsilon -\frac{\delta}{2\kappa_4}) \int_{\mathbb{R}^d}  |\nabla_x u|^2| \nabla_x V|^2 \rho_{\infty} dx\\
\leq \left[\frac{1}{2\delta} 
 +\frac{\delta}{4a\kappa_4}
+\frac{2 c^2_3a^2}{\varepsilon} \left(1+\frac{1}{2a\kappa_3}\right)\right]\int_{\mathbb{R}^d} w^2\rho_{\infty} dx.
\end{multline}
We choose $\delta$ and $\varepsilon$ such that $a-\varepsilon -\frac{\delta}{2\kappa_4}>0. $ Then, \eqref{varphi V6} shows that  \eqref{C_1} holds with $\displaystyle C_1 \colonequals \frac{1}{a-\varepsilon -\frac{\delta}{2\kappa_4}}\left[\frac{1}{2\delta} 
 +\frac{\delta}{4a\kappa_4}
+\frac{2 c_3^2a^2}{\varepsilon} \left(1+\frac{1}{2a\kappa_3}\right)\right]$.\\

Next, we  prove \eqref{C_2}:
We integrate by parts twice
\begin{align}\label{hess.^2}
\int_{\mathbb{R}^d}\left| \left|\frac{\partial^2 u}{\partial x^2} \right|\right|_F^2 \rho_{\infty} dx & =\sum_{i,j=1}^d \int_{\mathbb{R}^d} \partial^2_{x_i x_j} u\partial^2_{x_i x_j} u\rho_{\infty} dx\nonumber\\& =- \sum_{i,j=1}^d \int_{\mathbb{R}^d} \partial^3_{x_i x_jx_j} u \partial_{x_i } u \rho_{\infty} dx +\sum_{i,j=1}^d \int_{\mathbb{R}^d} \partial^2_{x_i x_j} u \partial_{x_i} u \partial_{x_j} V \rho_{\infty} dx\nonumber \\
&=\sum_{j=1}^d \int_{\mathbb{R}^d} \partial^2_{ x_jx_j} u\, \text{div}_x(\nabla_x u \rho_{\infty}) dx+ \int_{\mathbb{R}^d} \nabla_x^T u\frac{\partial^2 u}{\partial x^2}\nabla_x V \rho_{\infty} dx.
\end{align}
We use  \eqref{ellip.eq.gen} in the form $a\,  \text{div}_x(\nabla_x u \rho_{\infty})=(u-w)\rho_{\infty} $ and the H\"older inequality to estimate  \eqref{hess.^2}  
\begin{align*}
\int_{\mathbb{R}^d}\left|\left| \frac{\partial^2 u}{\partial x^2} \right|\right|_F^2 \rho_{\infty} dx &=a^{-1}\int_{\mathbb{R}^d} \Delta_x u (u-w)\rho_{\infty}  dx+ \int_{\mathbb{R}^d} \nabla_x^T u\frac{\partial^2 u}{\partial x^2} \nabla_x V \rho_{\infty}dx\nonumber \\ & \leq a^{-1}\sqrt{\int_{\mathbb{R}^d} |\Delta_x u|^2\rho_{\infty}  dx}\sqrt{\int_{\mathbb{R}^d} (u-w)^2\rho_{\infty}  dx}\nonumber \\ &\, \, \, \, \, \,\, \,  +\sqrt{\int_{\mathbb{R}^d} \left|  \left|\frac{\partial^2 u}{\partial x^2} \right|\right|_F^2  \rho_{\infty} dx}\sqrt{\int_{\mathbb{R}^d} |\nabla_x u|^2 | \nabla_x V|^2 \rho_{\infty} dx}.
\end{align*}
This inequality and $\displaystyle |\Delta_x u|^2\leq d \left|\left| \frac{\partial^2 
u}{\partial x^2} \right|\right|_F^2 $ show that 
\begin{equation}\label{hess. estimated}
  \sqrt{\int_{\mathbb{R}^d} \left| \left|\frac{\partial^2 
u}{\partial x^2} \right|\right|_F^2  \rho_{\infty} dx} \leq a^{-1}\sqrt{d\int_{\mathbb{R}^d} (u-w)^2\rho_{\infty}  dx} + 
\sqrt{\int_{\mathbb{R}^d} |\nabla_x u|^2 | \nabla_x V|^2 \rho_{\infty} dx}.
\end{equation}
Finally, \eqref{hess. estimated}, \eqref{grad varphi}, and \eqref{C_1} yield \eqref{C_2}.

\end{proof}

\subsection{Some matrix inequalities}
\begin{lemma}
Let $A\in \mathbb{R}^{d\times d}$ be a symmetric, positive definite matrix. For any $u, \, v \in \mathbb{R}^d,$ we have 
\begin{equation}\label{Mat.ineq}
2u^TAv\leq u^T Au+v^TAv \, \, \, \, \text{ and } \, \, \, \, 2u\cdot v\leq u^T Au+v^TA^{-1}v.
\end{equation}  
\end{lemma}
\begin{proof}
Since the matrices  $A\in \mathbb{R}^{d\times d}$ and  $\begin{pmatrix}
1&-1\\
-1& 1
\end{pmatrix}\in \mathbb{R}^{2\times 2}$ are positive semi-definite, their  Kronecker product 
$\begin{pmatrix}
A&-A\\-A&A
\end{pmatrix} $ is also positive semi-definite, see  \cite[Corollary 4.2.13]{MA}. Hence,
$$u^T Au+v^TAv-2u^TAv=\begin{pmatrix}
u\\v
\end{pmatrix}^T \begin{pmatrix}
A&-A\\-A&A
\end{pmatrix} \begin{pmatrix}
u\\v
\end{pmatrix}\geq 0.$$
The second inequality follows by replacing $v$ with $A^{-1}v.$
\end{proof}

We mention an alternative proof of the above lemma: If we consider $\mathbb{R}^d$ as a Hilbert space with the inner product $\langle u,v \rangle \colonequals u^TAv,$  then, in this Hilbert space,  the first matrix inequality in \eqref{Mat.ineq} coincides with Young's inequality.
 
\begin{lemma} Let $a_{ij}\colonequals \frac{\delta_{ij}+p_ip_j}{p_0},$ $p\in \mathbb{R}^d.$ Then
\begin{equation}\label{sum<d()}
\sum _{k,l,i,j=1}^d {p_0^2} \left(\delta_{kl}-\frac{p_kp_l}{p_0^2}\right) \left(\delta_{ij}-\frac{p_ip_j}{p_0^2}\right)\nabla_p a_{lj}\otimes  \nabla_p a_{ki}\leq d \left(I-\frac{p\otimes p}{p_0^2}\right) 
\end{equation}
holds for all $p\in \mathbb{R}^d.$
\end{lemma}
\begin{proof}
We compute  the element  which is in the intersection of  the $m^{\text{th}}$ column and the $n^{\text{th}}$ row:
\begin{multline*}
\sum_{k,l,i,j=1}^d {p_0^2}  \left(\delta_{kl}-\frac{p_kp_l}{p_0^2}\right)  \left(\delta_{ij}-\frac{p_ip_j}{p_0^2}\right) \partial_{p_m}a_{ki} \partial_{p_n}a_{lj}\\=\sum_{l,i =1}^d{p_0^2}  \left(\sum_{j=1}^d  \left(\delta_{ij}-\frac{p_ip_j}{p_0^2}\right)\partial_{p_n}a_{lj}\right)\left(\sum_{k=1}^d   \left(\delta_{kl}-\frac{p_kp_l}{p_0^2}\right)  \partial_{p_m}a_{ki}\right). 
\end{multline*}
We  first compute the sums in the brackets: 
\begin{align}\label{j sum}
\sum_{j=1}^d \left(\delta_{ij}-\frac{p_ip_j}{p_0^2}\right) \partial_{p_n}a_{lj}&=\sum_{j=1}^d \left(\delta_{ij}-\frac{p_ip_j}{p_0^2}\right)\left(\frac{p_l\delta_{nj}+p_j \delta_{nl}}{p_0}-\frac{(\delta_{lj}+p_lp_j)p_n}{p_0^3} \right)\nonumber \\
&=\frac{p_l\delta_{ni}+p_i\delta_{nl}}{p_0}-\frac{(\delta_{li}+p_lp_i)p_n}{p^3_0}-\frac{p_i}{p^2_0}\left(\frac{p_lp_n+|p|^2\delta_{nl}}{p_0}-\frac{(p_l+p_l|p|^2)p_n}{p^3_0}\right)\nonumber \\
&=\frac{p_l\delta_{ni}+p_i\delta_{nl}}{p_0}-\frac{(\delta_{li}+p_lp_i)p_n}{p^3_0}-\frac{p_ip_lp_n+p_i|p|^2\delta_{nl}}{p^3_0}+\frac{p_lp_ip_n}{p^3_0}\nonumber \\
&=\frac{p_l\delta_{ni}+p_i\delta_{nl}}{p_0}-\frac{(\delta_{li}+p_lp_i)p_n}{p^3_0}-\frac{p_i|p|^2\delta_{nl}}{p^3_0}\nonumber \\
&=\frac{p_l\delta_{ni}}{p_0}-\frac{(\delta_{li}+p_lp_i)p_n}{p^3_0}+\frac{p_i\delta_{nl}}{p^3_0}\nonumber \\
&=\frac{p_l\delta_{ni}}{p_0}-\frac{p_n\delta_{li}}{p_0^3}+\frac{p_i\delta_{nl}}{p^3_0}-\frac{p_lp_i p_n}{p^3_0}.
\end{align}
Similarly, we can show
\begin{equation}\label{k sum}\sum_{k=1}^d \left(\delta_{kl}-\frac{p_kp_l}{p_0^2}\right)  \partial_{p_m}a_{ki}=\frac{p_i\delta_{ml}}{p_0}-\frac{p_m\delta_{li}}{p_0^3}+\frac{p_l\delta_{mi}}{p^3_0}-\frac{p_lp_ip_m}{p^3_0}.
\end{equation}
Next, we sum the product of \eqref{j sum} and \eqref{k sum} with respect to $i$ and $l:$
$$\sum_{l,i=1}^d\left(\frac{p_l\delta_{ni}}{p_0}-\frac{p_n\delta_{li}}{p_0^3}+\frac{p_i\delta_{nl}}{p^3_0}-\frac{p_lp_i p_n}{p^3_0}\right)\left(\frac{p_i\delta_{ml}}{p_0}-\frac{p_m\delta_{li}}{p_0^3}+\frac{p_l\delta_{mi}}{p^3_0}-\frac{p_lp_ip_m}{p^3_0} \right)$$
$$=\sum_{l=1}^d\left[\frac{p_l}{p_0}\left(\frac{p_n\delta_{ml}}{p_0}-\frac{p_m\delta_{ln}}{p_0^3}+\frac{p_l\delta_{mn}}{p^3_0}-\frac{p_lp_np_m}{p^3_0}\right)-\frac{p_n}{p_0^3} \left(\frac{p_l\delta_{ml}}{p_0}-\frac{p_m}{p_0^3}+\frac{p_l\delta_{ml}}{p^3_0}-\frac{p^2_lp_m}{p^3_0} \right)\right.$$
$$\left.+\left(\frac{\delta_{nl}}{p^3_0}-\frac{p_l p_n}{p^3_0}\right)\left(\frac{|p|^2\delta_{ml}}{p_0}-\frac{p_mp_l}{p_0^3}+\frac{p_lp_m}{p^3_0}-\frac{p_lp_m|p|^2}{p^3_0} \right)\right]$$
$$=\left(\frac{p_np_m}{p^2_0}-\frac{p_mp_n}{p_0^4}+\frac{|p|^2\delta_{mn}}{p^4_0}-\frac{p_np_m|p|^2}{p^4_0}\right)- \left(\frac{p_np_m}{p^4_0}-\frac{dp_mp_n}{p_0^6}+\frac{p_mp_n}{p^6_0}-\frac{p_np_m|p|^2}{p^6_0} \right)$$
$$+\left(\frac{|p|^2\delta_{mn}}{p^4_0}-\frac{p_np_m|p|^2}{p^6_0} \right)-\left(\frac{p_np_m|p|^2}{p^4_0}-\frac{p_np_m|p|^4}{p^6_0} \right)$$
$$=\frac{2|p|^2\delta_{mn}}{p_0^4}+\frac{(d-2)p_mp_n}{p_0^6}-\frac{p_np_m|p|^2}{p^4_0}-\frac{p_np_m|p|^2}{p^6_0}+\frac{p_np_m|p|^4}{p^6_0}$$
$$=\frac{2\delta_{mn}}{p_0^2}-\frac{2\delta_{mn}}{p_0^4}+\frac{(d-2)p_mp_n}{p_0^6}-\frac{p_np_m}{p^2_0}+\frac{p_np_m}{p^4_0}+\frac{p_mp_n}{p_0^2}-\frac{3p_mp_n}{p_0^4}+\frac{2p_mp_n}{p_0^6}$$
$$=\frac{2\delta_{mn}}{p_0^2}-\frac{2\delta_{mn}}{p_0^4}+\frac{d p_mp_n}{p_0^6}-\frac{2p_mp_n}{p_0^4}=\frac{2}{p_0^2}(\delta_{mn}-\frac{p_np_m}{p_0^2})-\frac{2}{p_0^4}(\delta_{mn}-\frac{p_np_m}{p_0^2})+\frac{(d-2) p_mp_n}{p_0^6}.$$
This shows $$\sum _{k,l,i,j=1}^d {p_0^2} \left(\delta_{kl}-\frac{p_kp_l}{p_0^2}\right) \left(\delta_{ij}-\frac{p_ip_j}{p_0^2}\right)\nabla_p a_{lj}\otimes  \nabla_p a_{ki}=2\left(I-\frac{p \otimes p}{p_0^2}\right)-\frac{2}{p_0^2}\left(I-\frac{p\otimes p}{p_0^2}\right)+\frac{(d-2) p\otimes p}{p_0^4}.$$
The claimed inequality follows from 
$$\frac{(d-2) p\otimes p}{p_0^4}<\frac{d-2}{p_0^2}I\leq (d-2)\left(I-\frac{p \otimes p}{p_0^2}\right). $$
\end{proof}

\begin{lemma}
Let $a_{ij}\colonequals \frac{\delta_{ij}+p_ip_j}{p_0},$ $p\in \mathbb{R}^d.$ Then 
\begin{equation}\label{sum d(I+pop)}
\sum _{k,l,i,j=1}^d {p_0^2} \left(\delta_{kl}-\frac{p_kp_l}{p_0^2}\right) \left(\delta_{ij}-\frac{p_ip_j}{p_0^2}\right) ( (I+p\otimes p) \nabla_p a_{lj})\otimes   ((I+p\otimes p) \nabla_p a_{ki})\leq d (I+p\otimes p)
\end{equation}
holds for all $p\in \mathbb{R}^d.$
\end{lemma}
\begin{proof}
We compute  the element  which is in the intersection of  the $m^{\text{th}}$ column and the $n^{\text{th}}$ row:
$$\sum _{k,l,i,j=1}^d {p_0^2} \left(\delta_{kl}-\frac{p_kp_l}{p_0^2}\right) \left(\delta_{ij}-\frac{p_ip_j}{p_0^2}\right) ( \partial_{p_n}a_{lj}+p_n p\cdot \nabla_p a_{lj})(\partial_{p_m}a_{ki}+p_m p\cdot \nabla_p a_{ki})$$
$$=\sum_{l,i=1}^d p_0^2\left(\sum _{j=1}^d    \left(\delta_{ij}-\frac{p_ip_j}{p_0^2}\right) ( \partial_{p_n}a_{lj}+p_n p\cdot \nabla_p a_{lj})\right) \left( \sum_{k=1}^d \left(\delta_{kl}-\frac{p_kp_l}{p_0^2}\right)(\partial_{p_m}a_{ki}+p_m p\cdot \nabla_p a_{ki})\right).$$
We want to compute the sums in the brackets. We first compute
$$p\cdot \nabla_p a_{lj}=\sum_{r=1}^dp_r\left(\frac{p_l\delta_{rj}+p_j \delta_{rl}}{p_0}-\frac{(\delta_{lj}+p_lp_j)p_r}{p_0^3}\right)=\frac{2p_lp_j}{p_0}-\frac{(\delta_{lj}+p_lp_j)|p|^2}{p_0^3}$$
and so 
\begin{align*}\partial_{p_n}a_{lj}+p_n(p\cdot \nabla_p a_{lj})&=\frac{p_l\delta_{nj}+p_j \delta_{nl}}{p_0}-\frac{(\delta_{lj}+p_lp_j)p_n}{p_0^3}+\frac{2p_lp_jp_n}{p_0}-\frac{(\delta_{lj}+p_lp_j)p_n|p|^2}{p_0^3}\\
&=\frac{p_l\delta_{nj}+p_j \delta_{nl}-\delta_{lj}p_n}{p_0}+\frac{p_lp_jp_n}{p_0}.
\end{align*}
This helps us to compute 
$$\sum _{j=1}^d    \left(\delta_{ij}-\frac{p_ip_j}{p_0^2}\right) ( \partial_{p_n}a_{lj}+p_n p\cdot \nabla_p a_{lj})=\sum_{j=1}^d \left(\delta_{ij}-\frac{p_ip_j}{p_0^2}\right)\left(\frac{p_l\delta_{nj}+p_j \delta_{nl}-\delta_{lj}p_n}{p_0}+\frac{p_lp_jp_n}{p_0} \right)$$
$$=\frac{p_l\delta_{ni}+p_i \delta_{nl}-\delta_{li}p_n}{p_0}+\frac{p_lp_ip_n}{p_0}-\frac{p_i}{p_0^2}\left(\frac{|p|^2 \delta_{nl}}{p_0}+\frac{p_lp_n|p|^2}{p_0} \right)=\frac{p_l\delta_{ni}-\delta_{li}p_n}{p_0}+\frac{( \delta_{nl}+p_lp_n)p_i}{p^3_0}.$$
Similarly, we compute
$$\sum_{k=1}^d \left(\delta_{kl}-\frac{p_kp_l}{p_0^2}\right)(\partial_{p_m}a_{ki}+p_m p\cdot \nabla_p a_{ki})=\frac{p_i\delta_{ml}-\delta_{li}p_m}{p_0}+\frac{( \delta_{mi}+p_ip_m)p_l}{p^3_0}.$$
We sum the product of the last two equations with respect to $i$ and $l:$
$$\sum_{l,i=1}^d p_0^2\left(\sum _{j=1}^d    \left(\delta_{ij}-\frac{p_ip_j}{p_0^2}\right) ( \partial_{p_n}a_{lj}+p_n p\cdot \nabla_p a_{lj})\right) \left( \sum_{k=1}^d \left(\delta_{kl}-\frac{p_kp_l}{p_0^2}\right)(\partial_{p_m}a_{ki}+p_m p\cdot \nabla_p a_{ki})\right)$$
$$=\sum_{l,i=1}^d p_0^2\left(\frac{p_l\delta_{ni}-\delta_{li}p_n}{p_0}+\frac{( \delta_{nl}+p_lp_n)p_i}{p^3_0}\right)\left(\frac{p_i\delta_{ml}-\delta_{li}p_m}{p_0}+\frac{( \delta_{mi}+p_ip_m)p_l}{p^3_0}\right)$$
$$=\sum_{l=1}^d p_0^2\left[\frac{p_l}{p_0}\left(\frac{p_n\delta_{ml}-\delta_{ln}p_m}{p_0}+\frac{( \delta_{mn}+p_np_m)p_l}{p^3_0}\right)-\frac{p_n}{p_0}\left(\frac{p_l\delta_{ml}-p_m}{p_0}+\frac{( \delta_{ml}+p_lp_m)p_l}{p^3_0}\right)\right.$$
$$\left.+\frac{( \delta_{nl}+p_lp_n)}{p^3_0}\left(\frac{|p|^2\delta_{ml}-p_lp_m}{p_0}+\frac{( p_m+|p|^2p_m)p_l}{p^3_0}\right)\right]$$
$$=\frac{( \delta_{mn}+p_np_m)|p|^2}{p^2_0}+(d-1)p_mp_n-\frac{p_np_m+p_np_m|p|^2}{p^2_0}+\frac{|p|^2\delta_{mn}}{p^2_0}+\frac{p_np_m|p|^2}{p_0^2}$$
$$= d (\delta_{mn}+p_np_m)-(d-2)\delta_{mn} -\frac{2(\delta_{mn}+p_np_m)}{p^2_0}.$$
This shows 
$$\sum _{k,l,i,j=1}^d {p_0^2} \left(\delta_{kl}-\frac{p_kp_l}{p_0^2}\right) \left(\delta_{ij}-\frac{p_ip_j}{p_0^2}\right) ( (I+p\otimes p) \nabla_p a_{lj})\otimes   ((I+p\otimes p) \nabla_p a_{ki})$$
$$=d (I+p\otimes p)-(d-2)I -\frac{2(I+p\otimes p)}{p^2_0}\leq d (I+p\otimes p).$$
\end{proof}
\begin{lemma}\label{heavy comp} Let $P=P(x,p)$ be the matrix defined in \eqref{Pmat}. Then there are constant $\theta_1>0$ and $\theta_2>0$ such that 
$$\sum_{i,j=1}^d \frac{1}{f_{\infty}}\partial_{p_j}(\partial_{p_i}Pa_{ij}f_{\infty})\leq \begin{pmatrix}
\frac{2\theta_1\varepsilon^3}{V_0^3p_0^3} (I-\frac{p\otimes p}{p_0^2})& 0\\
0& \frac{2\theta_2\varepsilon}{V_0p_0} (I+p\otimes p)
\end{pmatrix}, \, \, \, \, \forall \, x, \,p\in \mathbb{R}^d.$$ 
\end{lemma}
\begin{proof} 
We have
\begin{equation}\label{5star}
\sum_{i,j=1}^d \frac{1}{f_{\infty}}\partial_{p_j}(\partial_{p_i}Pa_{ij}f_{\infty})=\sum_{i,j=1}^d \partial^2_{p_ip_j}P a_{ij} +\sum_{i=1}^d\partial_{p_i}P \sum_{j=1}^d\left(\partial_{p_j}a_{ij}-\frac{a_{ij}p_j}{p_0}\right).
\end{equation}
Since $a_{ij}=\frac{\delta_{ij}+p_ip_j}{p_0},$ we have  $$\sum_{j=1}^d\left(\partial_{p_j}a_{ij}-\frac{a_{ij}p_j}{p_0}\right)=\sum_{j=1}^d\left(\frac{p_i+\delta_{ij}p_j}{p_0}-\frac{\delta_{ij}+p_ip_j}{p_0^3}p_j-\frac{(\delta_{ij}+p_ip_j)p_j}{p^2_0}\right)=\left(\frac{d}{p_0}-1\right)p_i.$$ 
We denote $\varepsilon_1=\varepsilon_1(x)\colonequals \frac{\varepsilon}{V_0(x)}>0$ which is uniformly bounded for $x\in \mathbb{R}^d.$ Then
$$\partial_{p_i}P=\begin{pmatrix}
\frac{-6\varepsilon^3_1 p_i}{p_0^5} I -\frac{2\varepsilon^3_1 }{p_0^5}\partial_{p_i}(p\otimes p)+\frac{10\varepsilon^3_1p_i}{p_0^7}p\otimes p & -\frac{2\varepsilon^2_1 p_i}{p_0^4}I\\-\frac{2\varepsilon^2_1 p_i}{p_0^4}I& -\frac{2\varepsilon_1 p_i}{p_0^3} (I+p\otimes p)+\frac{2\varepsilon_1}{p_0}\partial_{p_i}(p\otimes p)
\end{pmatrix}.$$ 
The last two equations show
\begin{multline}\label{d/p-1}
\sum_{i=1}^d\partial_{p_i}P \sum_{j=1}^d\left(\partial_{p_j}a_{ij}-\frac{a_{ij}p_j}{p_0}\right)=\left(\frac{d}{p_0}-1\right)\sum_{i=1}^d\partial_{p_i}P p_i\\
=\left(\frac{d}{p_0}-1\right)\begin{pmatrix}
\frac{-6\varepsilon^3_1 |p|^2}{p_0^5} I -\frac{4\varepsilon^3_1 }{p_0^5}p\otimes p+\frac{10\varepsilon^3_1|p|^2}{p_0^7} p\otimes p&- \frac{2\varepsilon^2_1 |p|^2}{p_0^4}I\\-\frac{2\varepsilon^2_1 |p|^2}{p_0^4}I& -\frac{2\varepsilon_1 |p|^2}{p_0^3} (I+p\otimes p)+\frac{4\varepsilon_1}{p_0} p\otimes p
\end{pmatrix}\\
=\left(1-\frac{d}{p_0}\right)\begin{pmatrix}
\frac{6\varepsilon^3_1 }{p_0^3}( I-\frac{p \otimes p}{p_0^2}) +\frac{10\varepsilon^3_1 p\otimes p}{p_0^7}-\frac{6\varepsilon^3_1 }{p_0^5} I & \frac{2\varepsilon^2_1 }{p_0^2}I- \frac{2\varepsilon^2_1}{p_0^4}I\\ \frac{2\varepsilon^2_1 }{p_0^2}I- \frac{2\varepsilon^2_1}{p_0^4}I & -\frac{2\varepsilon_1 }{p_0} (I+p\otimes p)+\frac{4\varepsilon_1}{p_0}I-\frac{2\varepsilon_1 }{p_0^3}(I+p\otimes p)
\end{pmatrix}.
\end{multline}
One can easily check
\begin{equation}\label{aux.mat.ine}
\frac{p\otimes p}{p_0^2}<  I,\, \, \, \frac{1}{p_0^5}I\leq  \frac{1}{p_0^3}\left(I-\frac{p \otimes p}{p_0^2}\right), \, \, \, \frac{1}{p_0}I\leq \frac{1}{p_0}(I+p\otimes p).
\end{equation} 
These matrix inequalities help us to estimate  
\begin{equation}\label{d/p-1 6}\left(1-\frac{d}{p_0}\right)\left(\frac{6\varepsilon^3_1  }{p_0^3}( I-\frac{p \otimes p}{p_0^2}) +\frac{10\varepsilon^3_1 p\otimes p}{p_0^7}-\frac{6\varepsilon^3_1 }{p_0^5} I\right)\leq \frac{22\varepsilon^3_1 \left|1-\frac{d}{p_0}\right| }{p_0^3}\left( I-\frac{p \otimes p}{p_0^2}\right) 
\end{equation}
and 
\begin{equation}\label{d/p-1 2}\left(1-\frac{d}{p_0}\right)\left(-\frac{2\varepsilon_1 }{p_0} (I+p\otimes p)+\frac{4\varepsilon_1}{p_0}I-\frac{2\varepsilon_1 }{p_0^3}(I+p\otimes p)\right)\leq  \frac{8\varepsilon_1 \left|1-\frac{d}{p_0}\right| }{p_0}\left( I+p \otimes p\right). 
\end{equation}
\eqref{d/p-1}, \eqref{d/p-1 6}, and \eqref{d/p-1 2} imply
\begin{align}\label{2matin}\sum_{i=1}^d\partial_{p_i}P \sum_{j=1}^d\left(\partial_{p_j}a_{ij}-\frac{a_{ij}p_j}{p_0}\right)&\leq \begin{pmatrix}
\frac{22\varepsilon^3_1 \left|1-\frac{d}{p_0}\right| }{p_0^3}\left( I-\frac{p \otimes p}{p_0^2}\right)  & \left(1-\frac{d}{p_0}\right)\left(\frac{2\varepsilon^2_1 }{p_0^2}- \frac{2\varepsilon^2_1}{p_0^4}\right)I\\ \left(1-\frac{d}{p_0}\right)\left(\frac{2\varepsilon^2_1 }{p_0^2}- \frac{2\varepsilon^2_1}{p_0^4}\right)I & \frac{8\varepsilon_1 \left|1-\frac{d}{p_0}\right| }{p_0}\left( I+p \otimes p\right) 
\end{pmatrix}\nonumber\\
&\leq \begin{pmatrix}
\frac{24\varepsilon^3_1 \left|1-\frac{d}{p_0}\right| }{p_0^3}\left( I-\frac{p \otimes p}{p_0^2}\right)  & 0\\ 0 & \frac{10\varepsilon_1 \left|1-\frac{d}{p_0}\right| }{p_0}\left( I+p \otimes p\right) 
\end{pmatrix}.
\end{align}
In the last inequality, we used the fact that
\begin{equation*}
\begin{pmatrix}
-\frac{2\varepsilon^3_1\left|1-\frac{d}{p_0}\right|}{p_0^3}( I-\frac{p \otimes p}{p_0^2})  & \left(1-\frac{d}{p_0}\right)\left(\frac{2\varepsilon^2_1 }{p_0^2}- \frac{2\varepsilon^2_1}{p_0^4}\right)I\\ \left(1-\frac{d}{p_0}\right)\left(\frac{2\varepsilon^2_1 }{p_0^2}- \frac{2\varepsilon^2_1}{p_0^4}\right)I & -\frac{2\varepsilon_1 \left|1-\frac{d}{p_0}\right|}{p_0} (I+p\otimes p)
\end{pmatrix}
\end{equation*}
is negative semi-definite, which can be proven by using the second inequality in \eqref{Mat.ineq}. 
Since  $\left|1-\frac{d}{p_0}\right|$ is uniformly bounded for $p\in \mathbb{R}^d,$ \eqref{2matin} shows that   there are constants  $C_1>0$ and $C_2>0$ such that 
\begin{equation}\label{sum<P}\sum_{i=1}^d\partial_{p_i}P \sum_{j=1}^d\left(\partial_{p_j}a_{ij}-\frac{a_{ij}p_j}{p_0}\right)\leq  \begin{pmatrix}
\frac{2C_1\varepsilon^3_1}{p_0^{3}} (I-\frac{p\otimes p}{p_0^2})& 0\\
0& \frac{2C_2\varepsilon_1}{p_0} (I+p\otimes p)
\end{pmatrix}, \, \, \, \, \forall\, x,\, p\in \mathbb{R}^d.
\end{equation}
Next, using the computation above for $\partial_{p_i} P$ we compute 
$$\partial^2_{p_j p_i}P=\begin{pmatrix}
X_{ij}& -\frac{2\varepsilon^2_1 \delta_{ij}}{p_0^4}I+\frac{8\varepsilon^2_1 p_i p_j}{p_0^6}I\\-\frac{2\varepsilon^2_1 \delta_{ij}}{p_0^4}I+\frac{8\varepsilon^2_1 p_i p_j}{p_0^6}I&  Y_{ij}\end{pmatrix},$$
where
$$X_{ij}\colonequals -\frac{2\varepsilon^3_1 [3\delta_{ij} I+\partial_{p_ip_j}^2 (p \otimes p)]}{p_0^5}  +\frac{10\varepsilon^3_1 [3 p_i p_j I+ p_j\partial_{p_i}(p \otimes p)+p_i\partial_{p_j}(p \otimes p)+\delta_{ij} p \otimes p]}{p_0^7}-\frac{70 \varepsilon^3_1p_ip_j p\otimes p}{p_0^9}, $$
$$Y_{ij}\colonequals -2\varepsilon_1\left(\frac{ \delta_{ij}}{p_0^3}-\frac{3 p_i p_j}{p_0^5}\right) (I+p\otimes p)-\frac{2\varepsilon_1[p_i\partial_{p_j}(p \otimes p)+p_j\partial_{p_i}(p \otimes p)]}{p^3_0}+\frac{2\varepsilon_1}{p_0}\partial^2_{p_i p_j}(p\otimes p).
$$
The identities  
$$\sum_{i,j=1}^d \delta_{ij}\partial_{p_ip_j}^2 (p \otimes p)=2I,\,\,\, \sum_{i,j=1}^d  p_ip_j\partial_{p_ip_j}^2 (p \otimes p)=2 p \otimes p  , \, \, \,\, \sum_{i=1}^d  p_i\partial_{p_i}(p \otimes p)=2p\otimes p $$ will be used in the following computations:
$$X\colonequals \sum_{i,j=1}^d X_{ij}a_{ij}=-\sum_{i,j=1}^d\frac{2\varepsilon^3_1 [3\delta_{ij} I+\delta_{ij}\partial_{p_ip_j}^2 (p \otimes p)+3\delta_{ij}p_i p_j I+p_ip_j\partial_{p_ip_j}^2 (p \otimes p)]}{p_0^6}$$
$$+\sum_{i,j=1}^d \frac{10\varepsilon^3_1  [3\delta_{ij} p_i p_j I+ \delta_{ij} p_j\partial_{p_i}(p \otimes p)+\delta_{ij} p_i\partial_{p_j}(p \otimes p)+\delta_{ij} p \otimes p]}{p_0^8}$$
$$+\sum_{i,j=1}^d \frac{10\varepsilon^3_1 p_i p_j  [3 p_i p_j I+ p_j\partial_{p_i}(p \otimes p)+p_i\partial_{p_j}(p \otimes p)+\delta_{ij} p \otimes p]}{p_0^8}$$
$$-\sum_{i,j=1}^d \frac{70 \varepsilon^3_1[\delta _{ij} +p_ip_j ]p_ip_j p\otimes p}{p_0^{10}}$$
$$=-\frac{2\varepsilon^3_1 [3(d-1 +p_0^2)I+2(I+p \otimes p)]}{p_0^6}+ \frac{10\varepsilon^3_1  [3 |p|^2 p_0^2 I+4 p_0^2p \otimes p +(d-1+p_0^2)p \otimes p ]}{p_0^8}-\frac{70 \varepsilon^3_1 |p|^2 p\otimes p}{p_0^{8}}$$
$$=\frac{24 \varepsilon^3_1}{p_0^4}\left(I-\frac{p\otimes p}{p_0^2}\right)-\frac{ \varepsilon^3_1(28+6d)}{p_0^6}I+\frac{ \varepsilon^3_1(60+10d)}{p_0^8}p\otimes p,$$
$$Y \colonequals \sum_{i,j=1}^d Y_{ij}a_{ij}=-\sum_{i,j=1}^d 2\varepsilon_1\left(\frac{ \delta_{ij}}{p_0^4}+\frac{ \delta_{ij}p_ip_j}{p_0^4}-\frac{3\delta_{ij} p_i p_j}{p_0^6}-\frac{3 p^2_i p^2_j}{p_0^6}\right) (I+p\otimes p)$$
$$-\sum_{i,j=1}^d \frac{2\varepsilon_1  [\delta_{ij} p_i\partial_{p_j}(p \otimes p)+ \delta_{ij} p_j\partial_{p_i}(p \otimes p)]+2\varepsilon_1 p_i p_j[p_i\partial_{p_j}(p \otimes p)+p_j\partial_{p_i}(p \otimes p)]}{p^4_0}$$
$$+\sum_{i,j=1}^d \frac{2\varepsilon_1}{p^2_0}[\delta_{ij}\partial^2_{p_i p_j}(p\otimes p)+p_ip_j \partial^2_{p_i p_j}(p\otimes p)]
$$
$$=\left(\frac{ 4\varepsilon_1}{p_0^2}-\frac{2\varepsilon_1 (d+2) }{p_0^4}\right) (I+p\otimes p)- \frac{8 \varepsilon_1}{p_0^2} p\otimes p+ \frac{4\varepsilon_1}{p_0^2} (I+p\otimes p)$$
$$=\frac{8\varepsilon_1}{p_0^2}I- \frac{2 \varepsilon_1 (d+2)}{p_0^4}(I+p \otimes p),$$
$$Z \colonequals  \sum_{i,j=1}^d \left(-\frac{2\varepsilon_1^2 \delta_{ij}}{p_0^4}I+\frac{8\varepsilon_1^2 p_i p_j}{p_0^6}I\right)a_{ij}=-\sum_{i,j=1}^d\left(\frac{2\varepsilon^2_1 \delta_{ij}}{p_0^5}I-\frac{8\varepsilon^2_1 \delta_{ij} p_i p_j}{p_0^7}I+\frac{2\varepsilon^2_1 \delta_{ij}p_ip_j}{p_0^5}I-\frac{8\varepsilon^2_1 p^2_i p^2_j}{p_0^7}I\right)$$
$$=\frac{6\varepsilon^2_1 }{p_0^3}I-\frac{2\varepsilon^2_1 (d+3)}{p_0^5}I.$$
According to our notations, we have 
\begin{align}\label{24e_1^3/p_0^4}
\sum_{i,j=1}^d \partial^2_{p_ip_j}P a_{ij}&=\begin{pmatrix}
X&Z\\Z& Y
\end{pmatrix}\nonumber \\
&=\begin{pmatrix}
\frac{24 \varepsilon^3_1}{p_0^4}(I-\frac{p\otimes p}{p_0^2})-\frac{ \varepsilon^3_1(28+6d)}{p_0^6}I+\frac{ \varepsilon^3_1(60+10d)}{p_0^8}p\otimes p& \frac{6\varepsilon^2_1 }{p_0^3}I-\frac{2\varepsilon^2_1 (d+3)}{p_0^5}I\\ \frac{6\varepsilon^2_1 }{p_0^3}I-\frac{2\varepsilon^2_1 (d+3)}{p_0^5}I & \frac{8\varepsilon_1}{p_0^2}I- \frac{2 \varepsilon_1 (d+2)}{p_0^4}(I+p \otimes p)
\end{pmatrix}.
\end{align}
The matrix inequalities in \eqref{aux.mat.ine} help us to estimate
\begin{equation*}
\frac{24 \varepsilon^3_1}{p_0^4}\left(I-\frac{p\otimes p}{p_0^2}\right)-\frac{ \varepsilon^3_1(28+6d)}{p_0^6}I+\frac{ \varepsilon^3_1(60+10d)}{p_0^8}p\otimes p\leq \frac{ (84+10d)\varepsilon^3_1}{p_0^4}\left(I-\frac{p\otimes p}{p_0^2}\right)
\end{equation*}
and 
\begin{equation*}
\frac{8\varepsilon_1}{p_0^2}I- \frac{2 \varepsilon_1 (d+2)}{p_0^4}(I+p \otimes p)\leq \frac{8\varepsilon_1}{p_0^2}(I+p \otimes p).
\end{equation*}
These estimates and \eqref{24e_1^3/p_0^4} show
\begin{align}\label{84+10d}
\sum_{i,j=1}^d \partial^2_{p_ip_j}P a_{ij}
&\leq \begin{pmatrix}
\frac{ (84+10d)\varepsilon^3_1}{p_0^4}\left(I-\frac{p\otimes p}{p_0^2}\right)& \frac{6\varepsilon^2_1 }{p_0^3}I-\frac{2\varepsilon^2_1 (d+3)}{p_0^5}I\\ \frac{6\varepsilon^2_1 }{p_0^3}I-\frac{2\varepsilon^2_1 (d+3)}{p_0^5}I &  \frac{8\varepsilon_1}{p_0^2}(I+p \otimes p)
\end{pmatrix}\nonumber \\
&\leq \begin{pmatrix}
\frac{ \left(84+10d+\left|6-\frac{2(d+3)}{p_0^2}\right|\right)\varepsilon^3_1}{p_0^4}\left(I-\frac{p\otimes p}{p_0^2}\right)& 0\\ 0& &  \frac{\left(8+\left|6-\frac{2(d+3)}{p_0^2}\right|\right)\varepsilon_1}{p_0^2}(I+p \otimes p)
\end{pmatrix}.
\end{align} 
In the last inequality, we used the fact that
\begin{equation*}
\begin{pmatrix}
-\frac{ \left|6-\frac{2(d+3)}{p_0^2}\right|\varepsilon^3_1}{p_0^4}\left(I-\frac{p\otimes p}{p_0^2}\right)& \frac{6\varepsilon^2_1 }{p_0^3}I-\frac{2\varepsilon^2_1 (d+3)}{p_0^5}I\\ \frac{6\varepsilon^2_1 }{p_0^3}I-\frac{2\varepsilon^2_1 (d+3)}{p_0^5}I &  -\frac{\left|6-\frac{2(d+3)}{p_0^2}\right|\varepsilon_1}{p_0^2}(I+p \otimes p)
\end{pmatrix}
\end{equation*}
is negative semi-definite, which can be proven by using the second inequality in \eqref{Mat.ineq}. 
Since $\left|6-\frac{2(d+3)}{p_0^2}\right|$ is uniformly bounded for $p\in \mathbb{R}^d,$ \eqref{84+10d} shows that there are  constants $C_1'>0$ and $C_2'>0$  such that 
  \begin{align*} \sum_{i,j=1}^d \partial_{p_ip_j}P a_{ij}&\leq\begin{pmatrix}
\frac{2C_1'\varepsilon^3_1}{p_0^{4}} (I-\frac{p\otimes p}{p_0^2})& 0\\
0& \frac{2C_2'\varepsilon_1}{p_0^2} (I+p\otimes p)
\end{pmatrix}\\
&\leq\begin{pmatrix}
\frac{2C_1'\varepsilon^3_1}{p_0^{3}} (I-\frac{p\otimes p}{p_0^2})& 0\\
0& \frac{2C_2'\varepsilon_1}{p_0} (I+p\otimes p)
\end{pmatrix}, \, \, \, \, \forall p\in \mathbb{R}^d.
\end{align*}
 Using this estimate and \eqref{sum<P} in \eqref{5star}, we obtain the claimed result.
\end{proof}
\begin{lemma}\label{heavy comp2} Let $P=P(x,p)$ be the matrix defined in \eqref{Pmat}. Assume there exists a constant $c_3>0$ such that  \begin{equation}\label{VV}
 \left|\left|\frac{\partial^2 V(x)}{\partial x^2}\right|\right|_F\leq c_3(1+|\nabla_x V(x)|), \, \, \, \, \forall \, x \in \mathbb{R}^d.
\end{equation} Then there are constants $\theta_3>0$ and $\theta_4>0$ such that 
\begin{equation}\label{6star}\sum_{i=1}^d\left(\frac{p_i}{p_0} \partial_{x_i} P -\partial_{x_i} V  \partial_{p_i} P\right)\leq \begin{pmatrix}
\frac{2\theta_3\varepsilon^3}{V_0^2p_0^3} (I-\frac{p\otimes p}{p_0^2})& 0\\
0& \frac{2\theta_4\varepsilon}{p_0} (I+p\otimes p)
\end{pmatrix}, \, \, \, \, \forall \, x, \,p\in \mathbb{R}^d.
\end{equation}
\end{lemma}
\begin{proof} We compute
\begin{multline}\label{pixiP}\sum_{i=1}^d\frac{p_i}{p_0} \partial_{x_i} P=\left[\sum_{i=1}^d\frac{\nabla_x V \cdot \nabla_x (\partial_{x_i} V)p_i}{V_0^2p_0}\right]\begin{pmatrix}
-\frac{6\varepsilon^3}{V_0^3p_0^3} (I-\frac{p\otimes p}{p_0^2})& \frac{-2\varepsilon^2}{V_0^2p_0^2}I\\
 \frac{-2\varepsilon^2}{V_0^2p_0^2}I&  \frac{-2\varepsilon }{V_0 p_0} (I+p\otimes p)
\end{pmatrix}
\\
=\left[\frac{1}{V_0^2p_0}\nabla^T_x V \frac{\partial^2 V}{\partial x^2 }p\right]\begin{pmatrix}
-\frac{6\varepsilon^3}{V_0^3p_0^3} (I-\frac{p\otimes p}{p_0^2})& \frac{-2\varepsilon^2}{V_0^2p_0^2}I\\
 \frac{-2\varepsilon^2}{V_0^2p_0^2}I&  \frac{-2\varepsilon }{V_0 p_0} (I+p\otimes p)
\end{pmatrix}.
\end{multline}
 Since \eqref{VV} implies $\left|\left|\frac{\partial^2 V}{\partial x^2 }\right|\right|_{F}\leq \sqrt{2}c_3 V_0,$ we have  
 $$\left|\frac{1}{V_0^2p_0}\nabla^T_x V \frac{\partial^2 V}{\partial x^2 }p\right|\leq \frac{1}{V_0^2p_0}|\nabla_x V|\left|\left|\frac{\partial^2 V}{\partial x^2 }\right|\right|_{F}|p|\leq \sqrt{2}c_3,\, \, \, \, \forall \, x, \,p\in \mathbb{R}^d. $$
  This uniform bound and \eqref{pixiP} show that there are constants $C_1>0$ and $C_2$ such that
  \begin{equation}\label{pP}
  \sum_{i=1}^d\frac{p_i}{p_0} \partial_{x_i} P\leq \begin{pmatrix}
\frac{2C_1\varepsilon^3}{V_0^3p_0^3} (I-\frac{p\otimes p}{p_0^2})& 0\\
0& \frac{2C_2\varepsilon}{V_0p_0} (I+p\otimes p)
\end{pmatrix}\leq  \begin{pmatrix}
\frac{2C_1\varepsilon^3}{V_0^2p_0^3} (I-\frac{p\otimes p}{p_0^2})& 0\\
0& \frac{2C_2\varepsilon}{p_0} (I+p\otimes p)
\end{pmatrix}, \, \, \, \, \forall \, x, \,p\in \mathbb{R}^d.
  \end{equation}
  Next, we compute \begin{multline}\label{VP}
 - \sum_{i=1}^d\partial_{x_i} V  \partial_{p_i} P\\=\sum_{i=1}^d  \begin{pmatrix} 2\varepsilon^3\left[
\frac{5\partial_{x_i} V p_i}{V_0^3p_0^5} (I-\frac{p\otimes p}{p_0^2}) -\frac{\partial_{x_i} V }{V^3_0p_0^5}(2p_iI-\partial_{p_i}(p\otimes p)) \right]& \frac{2\varepsilon^2 \partial_{x_i} V p_i}{V^2_0p_0^4}I\\ \frac{2\varepsilon^2 \partial_{x_i} V p_i}{V_0^2p_0^4}I& \frac{2\varepsilon \partial_{x_i}  V p_i}{V_0p_0^3} (I+p\otimes p)-\frac{2\varepsilon \partial_{x_i} V}{V_0p_0}\partial_{p_i}(p\otimes p)
\end{pmatrix}.
  \end{multline} 
   We denote  $\tilde{p}_i\colonequals \begin{pmatrix}
  p_1\\\vdots\\p_{i-1}\\p_i-1\\p_{i+1}\\\vdots\\ p_d
  \end{pmatrix}=p-e_i$ and $\bar{p}_i\colonequals \begin{pmatrix}
  p_1\\\vdots\\p_{i-1}\\p_i+1\\p_{i+1}\\\vdots \\p_d
  \end{pmatrix}=p+e_i,$ where $e_i$ denotes the $i-$th unit vector, for $i\in \{1,..., d\}.$  Let $E_i\colonequals e_i\otimes e_i\in \mathbb{R}^{d\times d}$ denote the matrix whose element in the intersection of the $i-$th column and the $i-$th row equals 1 and all other elements are zero. Using $\partial_{p_i}(p\otimes p)=e_i\otimes p+p\otimes e_i$ one can check 
   $$\frac{1}{p_0^3} \left(I-\frac{p\otimes p}{p_0^2}\right)-\frac{1 }{p_0^5}(2p_iI-\partial_{p_i}(p\otimes p))=\frac{1}{p_0^5}E_i+ \frac{1 }{p_0^5} (|\tilde{p}_i|^2 I-\tilde{p}_i \otimes \tilde{p}_i)\geq 0$$
    and 
    $$\frac{1}{p_0^3} \left(I-\frac{p\otimes p}{p_0^2}\right)+\frac{1 }{p_0^5}(2p_iI-\partial_{p_i}(p\otimes p))=\frac{1}{p_0^5}E_i+ \frac{1 }{p_0^5} (|\bar{p}_i|^2 I-\bar{p}_i \otimes \bar{p}_i)\geq 0.$$
  From these equations we obtain
  $$-\frac{1}{p_0^3} \left(I-\frac{p\otimes p}{p_0^2}\right)\leq \frac{1 }{p_0^5}(2p_iI-\partial_{p_i}(p\otimes p))\leq \frac{1}{p_0^3} \left(I-\frac{p\otimes p}{p_0^2}\right).$$
    Using these inequalities and the fact that $\left|\frac{\partial_{x_i} V }{V_0}\right|$ and $\left|\frac{p_i}{p^2_0}\right|$ are uniformly bounded for $x, \, p\in \mathbb{R}^d,$ we conclude that there is a constant $C'_1>0$ such that 
    \begin{equation}\label{_11}2\varepsilon^3\left[
\frac{5\partial_{x_i} V p_i}{V_0^3p_0^5} \left(I-\frac{p\otimes p}{p_0^2}\right) -\frac{\partial_{x_i} V }{V^3_0p_0^5}(2p_iI-\partial_{p_i}(p\otimes p)) \right]\leq \frac{2C'_1\varepsilon^3}{V_0^2p_0^3} \left(I-\frac{p\otimes p}{p_0^2}\right), \, \, \,\, \, \forall x, p \in \mathbb{R}^d.
\end{equation}
The inequalities
$$I+p\otimes p-\partial_{p_i}(p\otimes p)= \tilde{p}_i \otimes \tilde{p}_i +I-E_i\geq 0 $$
and
$$I+p\otimes p+\partial_{p_i}(p\otimes p)= \bar{p}_i \otimes \bar{p}_i +I-E_i\geq 0 $$ are easy to check and they imply
$$-(I+p\otimes p)\leq \partial_{p_i}(p\otimes p)\leq  I+p\otimes p. $$
Using these inequalities and the fact that  $\left|\frac{\partial_{x_i} V }{V_0}\right|$ and $\left|\frac{p_i}{p^2_0}\right|$ are uniformly  bounded for $x, \, p \in \mathbb{R}^d,$ we conclude that there is a constant $C'_2>0$ such that 
\begin{equation}\label{_22}
\frac{2\varepsilon \partial_{x_i} V p_i}{V_0p_0^3} (I+p\otimes p)-\frac{2\varepsilon \partial_{x_i} V}{V_0p_0}\partial_{p_i}(p\otimes p)\leq \frac{2C'_2\varepsilon}{p_0} (I+p\otimes p).
\end{equation}
\eqref{VP}, \eqref{_11}, and \eqref{_22} show that there are constants $C''_1>0$ and $C''_2>0$ such that  
$$ - \sum_{i=1}^d\partial_{x_i} V  \partial_{p_i} P\leq \begin{pmatrix}
\frac{2C''_1\varepsilon^3}{V_0^2p_0^3} (I-\frac{p\otimes p}{p_0^2})& 0\\
0& \frac{2C''_2\varepsilon}{p_0} (I+p\otimes p)
\end{pmatrix}, \, \, \, \, \forall \, x, \,p\in \mathbb{R}^d.$$
This inequality and \eqref{pP} yield the claimed estimate \eqref{6star}.
\end{proof}
\textbf{Acknowledgement.} The  authors acknowledge support by the Austrian Science Fund (FWF)
project \href{https://doi.org/10.55776/F65}{10.55776/F65}. The second author is also funded  by the Deutsche Forschungsgemeinschaft (DFG, German Research Foundation) under Germany’s Excellence Strategy EXC 2044-390685587, Mathematics Münster: Dynamics–Geometry–Structure. \\

\textbf{Data Availability.} Data will be made available on reasonable request.\\

\textbf{Conflict of interest.} The authors have no conflict of interest to declare.

{}

\end{document}